\documentclass[12pt,oneside,english]{amsart}
\usepackage[T1]{fontenc}
\usepackage[latin9]{inputenc}
\usepackage{geometry}
\usepackage{mathrsfs}
\usepackage{amsthm}
\usepackage{amstext}
\usepackage{amssymb}
\usepackage{setspace}
\usepackage{esint}
\makeatletter
\numberwithin{equation}{section}
\numberwithin{figure}{section}
\theoremstyle{plain}
\newtheorem{thm}{\protect\theoremname}[section]
  \theoremstyle{definition}
  \newtheorem{defn}[thm]{\protect\definitionname}
  \theoremstyle{plain}
  \newtheorem{lem}[thm]{\protect\lemmaname}
  \theoremstyle{plain}
  \newtheorem{prop}[thm]{\protect\propositionname}
  \theoremstyle{remark}
  \newtheorem*{rem*}{\protect\remarkname}
  \theoremstyle{plain}
  \newtheorem{cor}[thm]{\protect\corollaryname}
  \theoremstyle{definition}
  \newtheorem{example}[thm]{\protect\examplename}

\makeatother

\usepackage{babel}
  \providecommand{\corollaryname}{Corollary}
  \providecommand{\definitionname}{Definition}
  \providecommand{\examplename}{Example}
  \providecommand{\lemmaname}{Lemma}
  \providecommand{\propositionname}{Proposition}
  \providecommand{\remarkname}{Remark}
\providecommand{\theoremname}{Theorem}

\begin{document}

\title{harmonic analysis \\for the bi-free partial $S$-transform}

\author{hao-wei huang and jiun-chau wang}

\date{Revised: April 27, 2017}



\address{Department of Applied Mathematics, National Sun Yat-sen University,
No. 70, Lienhai Road, Kaohsiung 80424, Taiwan, R. O. C.}

\email{hwhuang@math.nsysu.edu.tw}

\address{Department of Mathematics and Statistics, University of Saskatchewan,
106 Wiggins Road, Saskatoon, Saskatchewan S7N 5E6, Canada}

\email{jcwang@math.usask.ca}

\begin{abstract}
We develop an analytic machinery to study Voiculescu's bi-free partial
$S$-transform and then use the results to characterize the multiplicative
bi-free infinite divisibility. It is shown that the class of infinitely
divisible distributions coincides with the class of limit distributions
for products of bi-free pairs of left and right infinitesimal unitaries,
where the pairs are not required to be identically distributed but
all left variables are assumed to commute with all right variables.
Furthermore, necessary and sufficient conditions for convergence to
a given infinitely divisible distribution are found. 
\end{abstract}
\maketitle

\section{introduction}

The purpose of this paper is investigating the harmonic analysis for
the bi-free partial $S$-transform, an object that was introduced
by Voiculescu in \cite{part3} but has only been studied from the
combinatorial perspective so far (cf. \cite{Skoufranis}). 

The same task for the bi-free partial $R$-transform has been done
in our previous work \cite{HW} where the limit theory for sums of
bi-free partial $R$-transforms was established by harmonic and complex
analytic methods. This theory yields probability limit theorems for
the additive bi-free convolution of infinitesimal distributions on
$\mathbb{R}^{2}$, in which infinitely divisible measures arise naturally
as the limit distributions. In particular, these limit theorems provide
a way of constructing infinitely divisible laws in bi-free probability
theory. For example, the bi-free Gaussian measure%
\footnote{We remark that in {[}HW{]} the density formula for the bi-free Gaussian
measure contains an error; the coefficient for the $st$-term should
be $-2c(1+c^{2})$. We thank Paul Skoufranis for bringing this to
our attention.%
} was constructed through the bi-free analogue of the central limit
theorem in \cite{HW}. (See \cite{part1} for the origin of this measure,
as well as \cite{GHM} for an operator theoretical construction.)
The correspondence between limit theorems and their infinitely divisible
limits actually goes beyond the Gaussian case to every bi-freely infinitely
divisible measure, see \cite{HHW}. 

Bearing these results in mind, we start our course of investigation
by exploiting the analytic nature of the bi-free partial $S$-transform
$S_{\mu}$ of a probability measure $\mu$ on $\mathbb{T}^{2}=\left\{ (s,t)\in\mathbb{C}^{2}:|s|=1=|t|\right\} $.
We find it more convenient to do this with the function 
\[
\Sigma_{\mu}(z,w)=S_{\mu}(z/(1-z),\, w/(1-w)),
\]
called the \emph{$\Sigma$-transform}. Certainly, any result we proved
for $\Sigma_{\mu}$ can be easily translated to a statement about
$S_{\mu}$. Among all, Corollary 2.9 yields the identity 
\[
\Sigma_{\mu_{1}\boxtimes\boxtimes\mu_{2}}=\Sigma_{\mu_{1}}\Sigma_{\mu_{2}}
\]
in a circular Reinhardt domain $\Omega_{r}=\left\{ (z,w)\in\mathbb{C}^{2}:|z|,|w|\in[0,r)\cup(1/r,\infty)\right\} $,
where $r\in(0,1)$ and $\boxtimes\boxtimes$ denotes the multiplicative
bi-free convolution. When restricted to the bidisk component $D_{r}\times D_{r}=\left\{ (z,w)\in\mathbb{C}^{2}:|z|,|w|\in[0,r)\right\} $,
we recover Voiculescu's multiplicative identity for the bi-free partial
$S$-transform \cite{part3}. 

After building up analytic tools for the $\Sigma$-transform in Section
2, we proceed to study limit theorems and the infinitely divisibility
for $\boxtimes\boxtimes$ in Sections 3 and 4. The main results are
as follows. Unlike the theory of partial $R$-transform, neither $\Sigma_{\mu}$
nor $S_{\mu}$ alone can determine the underlying measure $\mu$.
Nevertheless, in Theorem 3.4 we manage to find the criteria for the
weak convergence of the bi-free convolutions 
\[
\delta_{\lambda_{n}}\boxtimes\boxtimes\mu_{n1}\boxtimes\boxtimes\mu_{n2}\boxtimes\boxtimes\cdots\boxtimes\boxtimes\mu_{nk_{n}},
\]
where $\delta_{\lambda_{n}}$ means the point mass at $\lambda_{n}\in\mathbb{T}^{2}$
and the array $\{\mu_{nk}:n\geq1,1\leq k\leq k_{n}\}$ of probability
measures on $\mathbb{T}^{2}$ is assumed to satisfy the \emph{infinitesimality}
condition 
\begin{equation}
\lim_{n\rightarrow\infty}\max_{1\leq k\leq k_{n}}\mu_{nk}(\{(s,t)\in\mathbb{T}^{2}:|s-1|+|t-1|\geq\varepsilon\})=0\label{eq:1.1}
\end{equation}
for any given $\varepsilon>0$. As to the infinite divisibility, Theorem
4.2 identifies bi-freely infinitely divisible measures as the weak
limits of infinitesimal arrays, making it possible to construct and
characterize a bi-freely infinitely divisible measure by the convergence
conditions in Theorem 3.4. Indeed, the multiplicative analogue of
the bi-free Gaussian and Poisson measures is constructed in this way,
see Examples 3.5 and 3.6. 

Last but not least, we mention that all left variables are assumed
to \emph{commute} with all right variables here, in order to accommodate
analytic objects such as measures and integral transforms.

\section{preliminaries}

\subsection{Bi-free convolution}

We first set up some notations. Denote by $\mathscr{M}_{\mathbb{T}^{2}}$
the family of all positive finite Borel measures on the distinguished
boundary $\mathbb{T}^{2}$ of the unit bidisk $\mathbb{D}^{2}=\left\{ (z,w)\in\mathbb{C}^{2}:|z|<1,|w|<1\right\} $,
and let $\mathscr{P}_{\mathbb{T}^{2}}$ be the subset of probability
measures in $\mathscr{M}_{\mathbb{T}^{2}}$. Analogously, the symbol
$\mathscr{M}_{\mathbb{T}}$ means the set of finite positive Borel
measures on the unit circle $\mathbb{T}$ and $\mathscr{P}_{\mathbb{T}}$
stands for the collection of probability measures in $\mathscr{M}_{\mathbb{T}}$. 

The above sets of measures are equipped with the weak-star topology
from duality with continuous functions on the underlying spaces. We
use the notation $\mu_{n}\Rightarrow\mu$ to indicate the convergence
of measures $\mu_{n}$ to $\mu$ in this weak topology as $n\rightarrow\infty$.
By Prohorov's theorem, any sequence of probability measures on $\mathbb{T}$
(respectively, $\mathbb{T}^{2}$) is tight and hence has a subsequence
converging weakly to a probability measure on $\mathbb{T}$ (resp.,
$\mathbb{T}^{2}$). Also, note that the weak convergence of measures
is equivalent to convergence in moments. 

Consider the $C^{*}$-probability space $B(H)$ of bounded linear
operators acting on a Hilbert space $H$, where the expectation is
given by $\varphi_{\xi}(\cdot)=\left\langle \cdot\xi,\xi\right\rangle $
for some unit vector $\xi\in H$. Given two commuting unitary operators
$u,v\in B(H)$, let $A$ be the commutative $C^{*}$-algebra generated
by the identity operator $I$, $u$, and $v$ in $B(H)$. If $M$
denotes the maximal ideal space of $A$, then the Gelfand map $\Phi$
is an isometric $*$-isomorphism from $A$ onto the $C^{*}$-algebra
$C(M)$ of continuous functions on the compact set $M$. Define a
map $F(m)=\left(\Phi(u)(m),\Phi(v)(m)\right)$ for $m\in M$ and let
$X=F(M)\subset\mathbb{\mathbb{T}}^{2}$. Then the homeomorphism $F:M\rightarrow X$
provides a complex coordinate chart on $M$, and we obtain a $*$-representation
$\pi:C(X)\rightarrow B(H)$ by $\pi(f)=\Phi^{-1}\left(f\circ F\right)$.
The representation $\pi$ gives rise to a spectral measure $E_{(u,v)}$
such that the continuous functional calculus $f(u,v)$ can be written
as 
\[
f(u,v)=\int f\, dE_{(u,v)}
\]
for any $f\in C(\mathbb{T}^{2})$. The \emph{distribution} $\mu_{(u,v)}$
of the pair $(u,v)$ is then defined as the probability measure $\varphi_{\xi}\circ E_{(u,v)}$
in $\mathscr{P}_{\mathbb{T}^{2}}$. 

More generally, the distribution for any pair of commuting unitary
variables in a $C{}^{*}$-probability space $\left(\mathcal{A},\varphi\right)$
is defined in the same way, provided that we first represent the algebra
$\mathcal{A}$ on a Hilbert space and realize the expectation $\varphi$
as a vector state through the GNS construction. 

For any $\mu\in\mathscr{M}_{\mathbb{T}^{2}}$, let $\mu^{(1)}$ and
$\mu^{(2)}$ denote respectively the push-forward of the measure $\mu$
under the continuous coordinate projections $\pi_{1}(s,t)=s$ and
$\pi_{2}(s,t)=t$, that is, $\mu^{(j)}=\mu\circ\pi_{j}^{-1}$ for
$j=1,2$. The measures $\mu^{(1)}$, $\mu^{(2)}$ are called the \emph{marginal
laws} of $\mu$ and belong to $\mathscr{M}_{\mathbb{T}}$. If $\mu=\mu_{(u,v)}$
for some commuting unitaries $u$, $v$ in a $C^{*}$-probability
space $(\mathcal{A},\varphi)$, then the marginal laws $\mu^{(1)}$
and $\mu^{(2)}$ are the distributions of the variables $u$ and $v$,
respectively. Thus, one has 
\[
\int_{\mathbb{T}}f(s)\, d\mu^{(1)}(s)=\varphi(f(u))=\int_{\mathbb{T}^{2}}f(s)\, d\mu(s,t)
\]
and 
\[
\int_{\mathbb{T}}g(t)\, d\mu^{(2)}(t)=\varphi(g(v))=\int_{\mathbb{T}^{2}}g(t)\, d\mu(s,t)
\]
for any $f,g\in C(\mathbb{T})$. 

We now follow \cite{part1} to present a construction of the (multiplicative)
bi-free convolution of two probability measures on $\mathbb{T}^{2}$.
Fix $\mu_{1},\mu_{2}\in\mathscr{P}_{\mathbb{T}^{2}}$ and consider
their $L^{2}$-Hilbert spaces $H_{i}=L^{2}(\mu_{i})$ for $i=1,2$.
Let $\xi_{i}$ be the constant function one in $H_{i}$. According
to \cite{part1}, the identification of the space $H_{i}$ as the
left (resp., the right) tensor factor of the free product Hilbert
space $(H,\xi)=(H_{1},\xi_{1})*(H_{2},\xi_{1})$ is done by a Hilbert
space isomorphism $V_{i}:H_{i}\otimes H(\ell,i)\rightarrow H$ (resp.,
$W_{i}:H(r,i)\otimes H_{i}\rightarrow H$). We refer to \cite{part1}
for the definition of the maps $V_{i}$, $W_{i}$ and that of the
spaces $H(\ell,i)$, $H(r,i)$. Here we only recall the facts that
$V_{i}(\xi_{i}\otimes\xi)=\xi=W_{i}(\xi\otimes\xi_{i})$ and $V_{i}(h_{i}\otimes\xi)=W_{i}(\xi\otimes h_{i})$
for all $h_{i}\in H_{i}$. 

For $T\in B(H_{i})$, define the \emph{left} and \emph{right} \emph{variables}
by the formulae $\lambda_{i}(T)=V_{i}(T\otimes I_{H(\ell,i)})V_{i}^{-1}$
and $\rho_{i}(T)=W_{i}(I_{H(r,i)}\otimes T)W_{i}^{-1}$ respectively.
Then $\lambda_{i},\rho_{i}:B(H_{i})\rightarrow B(H)$ are $*$-representations,
and in the sense of Voiculescu \cite{part1}, the pairs of $C^{*}$-faces
$\left(B(H_{1}),B(H_{1})\right)$ and $\left(B(H_{2}),B(H_{2})\right)$
are \emph{bi-free} in the $C^{*}$-probability space $(B(H),\varphi_{\xi})$
via the homomorphisms $\beta_{i}=\lambda_{i}$ and $\gamma_{i}=\rho_{i}$
($i=1,2$). Indeed, this follows from the definition of bi-freeness
by choosing the vector spaces $\mathcal{X}_{i}=H_{i}$, $\mathcal{X}_{i}^{\circ}=H_{i}\ominus\mathbb{C}\xi_{i}$
and the homomorphisms $l_{i}(T)=r_{i}(T)=T$ for $T\in B(H_{i})$
where $i=1,2$. In particular, this implies that for any $S_{1},T_{1}\in B(H_{1})$
and $S_{2},T_{2}\in B(H_{2})$, the two-faced pairs 
\[
\left(\lambda_{1}(S_{1}),\rho_{1}(T_{1})\right)\quad\text{and}\quad\left(\lambda_{2}(S_{2}),\rho_{2}(T_{2})\right)
\]
of left and right variables in $(B(H),\varphi_{\xi})$ are bi-free
in the $C^{*}$-setting (cf. \cite[Section 3]{part1}). 

Note that the left variables $\left\{ \lambda_{1}(S_{1}),\lambda_{2}(S_{2})\right\} $,
as well as the right variables alone, are free among themselves in
the space $(B(H),\varphi_{\xi})$.

Now, it is fairly easy to construct bi-free random vectors having
distributions $\mu_{1}$ and $\mu_{2}$. For $i=1,2$, define commuting
unitary operators $S_{i},T_{i}\in B(H_{i})$ by 
\[
S_{i}(f)(s,t)=sf(s,t)\quad\text{and}\quad T_{i}(f)(s,t)=tf(s,t),\quad(s,t)\in\mathbb{T}^{2},\; f\in H_{i},
\]
so that the distribution of $(S_{i},T_{i})$ with respect to the expectation
$\varphi_{\xi_{i}}(\cdot)=\left\langle \cdot\xi_{i},\xi_{i}\right\rangle $
is the measure $\mu_{i}$. We need to show that the pair $(\lambda_{i}(S_{i}),\rho_{i}(T_{i}))$
of left and right unitary variables in $(B(H),\varphi_{\xi})$ also
has the distribution $\mu_{i}$. To this end, first, the commutation
relation 
\begin{equation}
\lambda_{i}(S_{i})\rho_{i^{\prime}}(T_{i^{\prime}})=\rho_{i^{\prime}}(T_{i^{\prime}})\lambda_{i}(S_{i}),\quad i,i^{\prime}\in\{1,2\},\label{eq:2.1}
\end{equation}
from \cite[Section 1.5]{part1} shows that the distribution of $(\lambda_{i}(S_{i}),\rho_{i}(T_{i}))$
is well-defined as a measure in $\mathscr{P}_{\mathbb{T}^{2}}$. Secondly,
we compute the moment 
\begin{eqnarray*}
\varphi_{\xi}\left(\lambda_{i}(S_{i})^{p}\rho_{i}(T_{i})^{q}\right) & = & \left\langle \lambda_{i}(S_{i}^{p})\rho_{i}(T_{i}^{q})\xi,\xi\right\rangle \\
 & = & \left\langle W_{i}(I\otimes T_{i}^{q})(\xi\otimes\xi_{i}),V_{i}(S_{i}^{-p}\otimes I)(\xi_{i}\otimes\xi)\right\rangle \\
 & = & \left\langle W_{i}(\xi\otimes T_{i}^{q}\xi_{i}),V_{i}(S_{i}^{-p}\xi_{i}\otimes\xi)\right\rangle \\
 & = & \left\langle W_{i}(\xi\otimes T_{i}^{q}\xi_{i}),W_{i}(\xi\otimes S_{i}^{-p}\xi_{i})\right\rangle \\
 & = & \left\langle \xi\otimes T_{i}^{q}\xi_{i},\xi\otimes S_{i}^{-p}\xi_{i}\right\rangle =\varphi_{\xi_{i}}(S_{i}^{p}T_{i}^{q})=\int_{\mathbb{T}^{2}}s^{p}t^{q}\, d\mu_{i}(s,t)
\end{eqnarray*}
for any $p,q\in\mathbb{Z}$ to get the desired relation 
\[
\mu_{(\lambda_{i}(S_{i}),\rho_{i}(T_{i}))}=\mu_{i}.
\]

In summary, denoting $u_{i}=\lambda_{i}(S_{i})$ and $v_{i}=\rho_{i}(T_{i})$
for $i=1,2$, the pairs $(u_{1},v_{1})$ and $(u_{2},v_{2})$ in $(B(H),\varphi_{\xi})$
are bi-free and distributed in the right way. This leads to the following
definition. 
\begin{defn}
For any $\mu_{1},\mu_{2}\in\mathscr{P_{\mathbb{T}^{2}}}$, their \emph{bi-free
convolution} $\mu_{1}\boxtimes\boxtimes\mu_{2}$ is the distribution
of the product $(u_{1}u_{2},v_{1}v_{2})$. 
\end{defn}
Several remarks are in order. First, by (\ref{eq:2.1}), the unitary
variables $u_{1}u_{2}$ and $v_{1}v_{2}$ commute with each other,
so that the bi-free convolution $\mu_{1}\boxtimes\boxtimes\mu_{2}$
is indeed a probability measure on $\mathbb{T}^{2}$. 

Second, the measure $\mu_{1}\boxtimes\boxtimes\mu_{2}$ is completely
determined by $\mu_{1}$ and $\mu_{2}$. This is because the bi-freeness
for $(u_{1},v_{1})$ and $(u_{2},v_{2})$ implies that any joint $*$-moment
of $\{u_{1}u_{2},v_{1}v_{2}\}$ is given by a universal polynomial
relation among the joint $*$-moments of $\{u_{1},v_{1}\}$ and that
of $\{u_{2},v_{2}\}$. (By virtue of (\ref{eq:2.1}), these joint
$*$-moments can be reduced to the usual moments.) More precisely,
given $p,q\in\mathbb{Z}$ with $|p|+|q|\geq1$ and let $m$ be the
number of pairs $(p^{\prime},q^{\prime})\in\mathbb{Z}^{2}$ satisfying
$1\leq|p^{\prime}|+|q^{\prime}|\leq|p|+|q|$, Proposition 2.18 and
Lemma 5.2 from \cite{part1} imply that there exists a universal polynomial
$P_{p,q}\in\mathbb{Z}[X_{1},X_{2},\cdots,X_{m},Y_{1},Y_{2},\cdots,Y_{m}]$
such that the corresponding moment $\varphi_{\xi}((u_{1}u_{2})^{p}(v_{1}v_{2})^{q})$
is equal to the value of the polynomial $P_{p,q}$ at $X_{i}=\varphi_{\xi}(u_{1}^{p_{i}}v_{1}^{q_{i}})$
and $Y_{i}=\varphi_{\xi}(u_{2}^{p_{i}}v_{2}^{q_{i}})$, where $1\leq|p_{i}|+|q_{i}|\leq|p|+|q|$
for $1\leq i\leq m$. 

Another consequence of this universality is that the commutative,
associative binary operation $\boxtimes\boxtimes$ on the set $\mathscr{P}_{\mathbb{T}^{2}}$
is jointly weak-star continuous in the sense that $\mu_{n}\boxtimes\boxtimes\nu_{n}\Rightarrow\mu\boxtimes\boxtimes\nu$
in $\mathscr{P}_{\mathbb{T}^{2}}$ whenever $\mu_{n}\Rightarrow\mu$,
$\nu_{n}\Rightarrow\nu$ in $\mathscr{P}_{\mathbb{T}^{2}}$.

Here we mention a special case of the polynomial $P_{p,q}$ when all
involved random variables have expectation zero. The result below
is Lemma 2.1 from \cite{part2}.
\begin{lem}
Let $\mathcal{A}_{i}=C^{*}(u_{i})$ and $\mathcal{B}_{i}=C^{*}(v_{i})$
be the $C^{*}$-subalgebras generated by $u_{i}$ and $v_{i}$ in
$(B(H),\varphi_{\xi})$, respectively. Let $1\leq k\leq m$, $1\leq l\leq n$,
$\alpha(k),\beta(l)\in\{1,2\}$, $a_{k}\in\mathcal{A}_{\alpha(k)}$,
$b_{l}\in\mathcal{B}_{\beta(l)}$ be such that $\varphi_{\xi}(a_{k})=0=\varphi_{\xi}(b_{l})$
and $\alpha(1)\neq\alpha(2)\neq\cdots\neq\alpha(m)$, $\beta(1)\neq\beta(2)\neq\cdots\neq\beta(n)$.
Then we have 
\[
\varphi_{\xi}(a_{m}a_{m-1}\cdots a_{1}b_{n}b_{n-1}\cdots b_{1})=\delta_{m,n}\prod_{1\leq k\leq\min\{m,n\}}\delta_{\alpha(k),\beta(k)}\,\varphi_{\xi}(a_{k}b_{k}).
\]

\end{lem}
Denote by 
\[
\text{m}=\frac{d\theta}{2\pi}\otimes\frac{d\theta}{2\pi}
\]
the uniform distribution on $\mathbb{T}^{2}$, the next result follows
from Lemma 2.2. 
\begin{prop}
Let $\mu_{i}=\mu_{(u_{i},v_{i})}$ and assume that $\varphi_{\xi}(u_{i})=0=\varphi_{\xi}(v_{i})$
for $i=1,2$. Then we have 
\[
\mu_{1}\boxtimes\boxtimes\mu_{2}=\text{\emph{m}}
\]
if and only if one of the moments $\varphi_{\xi}(u_{1}v_{1})$, $\varphi_{\xi}(u_{2}v_{2})$
is zero. In particular, we always have $\mu_{(u,v)}\boxtimes\boxtimes\emph{m}=\emph{m}$
as long as $\varphi_{\xi}(u)=0=\varphi_{\xi}(v)$.\end{prop}
\begin{proof}
We need to show that $\varphi_{\xi}((u_{1}u_{2})^{p}(v_{1}v_{2})^{q})=0$
holds for any given $p,q\in\mathbb{Z}$, $|p|+|q|\geq1$, if and only
if $\varphi_{\xi}(u_{1}v_{1})=0$ or $\varphi_{\xi}(u_{2}v_{2})=0$.
To this purpose, consider two such integers $p$ and $q$. First,
if $p$ or $q$ happens to be zero, for example, say $q=0$, then
the freeness of $\{u_{1},u_{2}\}$ implies $\varphi_{\xi}((u_{1}u_{2})^{p})=0$
without any extra conditions. So, we may and do assume that $|p|+|q|\geq2$.
Then, depending on the signs of $p$ and $q$, we have 
\[
\varphi_{\xi}((u_{1}u_{2})^{p}(v_{1}v_{2})^{q})=\begin{cases}
\varphi_{\xi}(u_{1}u_{2}\cdots u_{1}u_{2}v_{1}v_{2}\cdots v_{1}v_{2}), & p,q\geq1;\\
\varphi_{\xi}(u_{1}u_{2}\cdots u_{1}u_{2}v_{2}^{*}v_{1}^{*}\cdots v_{2}^{*}v_{1}^{*}), & p\geq1,q\leq-1;\\
\varphi_{\xi}(u_{2}^{*}u_{1}^{*}\cdots u_{2}^{*}u_{1}^{*}v_{1}v_{2}\cdots v_{1}v_{2}), & p\leq-1,q\geq1;\\
\varphi_{\xi}(u_{2}^{*}u_{1}^{*}\cdots u_{2}^{*}u_{1}^{*}v_{2}^{*}v_{1}^{*}\cdots v_{2}^{*}v_{1}^{*}), & p,q\leq-1,
\end{cases}
\]
where the number of the left variables $u_{i}$'s is $2|p|$ and that
of the right variables $v_{i}$'s is $2|q|$. Lemma 2.2 shows that
we can exclude the case $|p|\neq|q|$. Moreover, even when $|p|=|q|$,
the second and the third types of $*$-moments are always zero. Thus,
by Lemma 2.2 again, we end up with the next two cases after pairing
the left and right variables: 
\[
\varphi_{\xi}((u_{1}u_{2})^{p}(v_{1}v_{2})^{q})=\begin{cases}
\varphi_{\xi}(u_{1}v_{1})^{p}\varphi_{\xi}(u_{2}v_{2})^{q}, & p=q\geq1;\\
\left[\overline{\varphi_{\xi}(u_{1}v_{1})}\right]^{|p|}\left[\overline{\varphi_{\xi}(u_{2}v_{2})}\right]^{|q|}, & p=q\leq-1.
\end{cases}
\]
The desired result now follows immediately.
\end{proof}
For the purpose of this paper, we introduce another convolution operation
on $\mathscr{P}_{\mathbb{T}^{2}}$. 
\begin{defn}
The (right) \emph{opposite bi-free convolution} $\mu_{1}\boxtimes\boxtimes^{\text{op}}\mu_{2}$
of $\mu_{1}$ and $\mu_{2}$ is a probability measure on $\mathbb{T}^{2}$
defined by 
\[
\mu_{1}\boxtimes\boxtimes^{\text{op}}\mu_{2}=\mu_{(u_{1}u_{2},v_{2}v_{1})}.
\]

\end{defn}
It is easy to see that $\mu_{1}\boxtimes\boxtimes\mu_{2}\neq\mu_{1}\boxtimes\boxtimes^{\text{op}}\mu_{2}$,
for example, through Lemma 2.2. Nevertheless, if we introduce the
(right) \emph{coordinate-reflection} $\mu^{*}$ of a measure $\mu\in\mathscr{P}_{\mathbb{T}^{2}}$
by 
\[
d\mu^{*}(s,t)=d\mu(s,1/t),
\]
then we do have 
\begin{eqnarray*}
(\mu_{1}\boxtimes\boxtimes\mu_{2})^{*} & = & \mu_{(u_{1}u_{2},(v_{1}v_{2})^{*})}\\
 & = & \mu_{(u_{1},v_{1}^{*})}\boxtimes\boxtimes^{\text{op}}\mu_{(u_{2},v_{2}^{*})}\\
 & = & \mu_{1}^{*}\boxtimes\boxtimes^{\text{op}}\mu_{2}^{*}.
\end{eqnarray*}

Finally, from the construction of the convolutions $\boxtimes\boxtimes$
and $\boxtimes\boxtimes^{\text{op}}$, we know that 
\[
(\mu_{1}\boxtimes\boxtimes\mu_{2})^{(j)}=\mu_{1}^{(j)}\boxtimes\mu_{2}^{(j)},\qquad j=1,2,
\]
and 
\[
(\mu_{1}\boxtimes\boxtimes^{\text{op}}\mu_{2})^{(1)}=\mu_{1}^{(1)}\boxtimes\mu_{2}^{(1)},\quad(\mu_{1}\boxtimes\boxtimes^{\text{op}}\mu_{2})^{(2)}=\mu_{2}^{(2)}\boxtimes\mu_{1}^{(2)}=\mu_{1}^{(2)}\boxtimes\mu_{2}^{(2)}.
\]

\subsection{The $\psi$- and $\eta$-transforms}

The \emph{$\psi$-transform} $\psi_{\mu}$ for $\mu\in\mathscr{P}_{\mathbb{T}^{2}}$
is defined by the integral 
\[
\psi_{\mu}(z,w)=\int_{\mathbb{T}^{2}}\frac{zs}{1-zs}\frac{wt}{1-wt}\, d\mu(s,t)
\]
for $(z,w)$ in the product set $(\mathbb{C}\setminus\mathbb{T})^{2}=(\mathbb{C}\setminus\mathbb{T})\times(\mathbb{C}\setminus\mathbb{T})=\left\{ (z,w)\in\mathbb{C}^{2}:|z|\neq1\neq|w|\right\} $.
The \emph{one-dimensional (1-D) $\psi$-transform} $\psi_{\mu^{(j)}}$
for the marginal law $\mu^{(j)}$ or, more generally, for any probability
measure on $\mathbb{T}$, is defined as 
\[
\psi_{\mu^{(j)}}(z)=\int_{\mathbb{T}}\frac{zx}{1-zx}\, d\mu^{(j)}(x),\qquad z\in\mathbb{C}\setminus\mathbb{T}.
\]
The integral transforms $\psi_{\mu}$ and $\psi_{\mu^{(j)}}$ are
holomorphic and satisfy the symmetries:
\[
\begin{cases}
\psi_{\mu^{(j)}}(z)+1=-\overline{\psi_{\mu^{(j)}}(1/\overline{z})}, & z\in\mathbb{C}\setminus\mathbb{T},\; j=1,2;\\
\psi_{\mu}(z,w)+\psi_{\mu^{(1)}}(z)+\psi_{\mu^{(2)}}(w)+1=\overline{\psi_{\mu}(1/\overline{z},1/\overline{w})}, & (z,w)\in(\mathbb{C}\setminus\mathbb{T})^{2}.
\end{cases}
\]
We also note the following limits: $\psi_{\mu^{(j)}}(\infty)=\lim_{|z|\rightarrow\infty}\psi_{\mu^{(j)}}(z)=-1$,
\[
\psi_{\mu}(z,\infty)=\lim_{|w|\rightarrow\infty}\psi_{\mu}(z,w)=-\psi_{\mu^{(1)}}(z),\quad\psi_{\mu}(\infty,w)=\lim_{|z|\rightarrow\infty}\psi_{\mu}(z,w)=-\psi_{\mu^{(2)}}(w),
\]
and $\psi_{\mu}(\infty,\infty)=\lim_{|z|,\,|w|\rightarrow\infty}\psi_{\mu}(z,w)=1$. 

From the spectral theorem, if $\mu=\mu_{(u,v)}$ for some commuting
unitaries $u,v$ in a $C^{*}$-probability space $(\mathcal{A},\varphi)$
then the $\psi$-transforms are merely the expectation of resolvent
type functions in $u$ and $v$; for examples, $\psi_{\mu}(z,w)=zw\varphi(uv[1-zu]^{-1}[1-wv]^{-1})$
or $\psi_{\mu^{(1)}}(z)=z\varphi(u[1-zu]^{-1})$. For this reason,
in this paper we sometimes use the variables $u,v$ to label the relevant
transforms instead of using $\mu$ or $\mu^{(j)}$. 

The $\psi$-transforms determine the underlying measures. For a marginal
law $\mu^{(j)}$, this can be seen from the fact that the real part
of 
\[
2\psi_{\mu^{(j)}}(z)+1=\int_{\mathbb{T}}\frac{1+zx}{1-zx}\, d\mu^{(j)}(x),\qquad z\in\mathbb{D},
\]
is the Poisson integral of the measure $d\mu^{(j)}(1/x)$. On the
other hand, observe that 
\[
g(z,w)=4\psi_{\mu}(z,w)+2[\psi_{\mu^{(1)}}(z)+\psi_{\mu^{(2)}}(w)]+1=\int_{\mathbb{T}^{2}}\frac{1+zs}{1-zs}\frac{1+wt}{1-wt}\, d\mu(s,t)
\]
for $(z,w)\in(\mathbb{C}\setminus\mathbb{T})^{2}$ so that the formula
\begin{equation}
\Re\left[\frac{g(z,w)-g(z,1/\overline{w})}{2}\right]=\int_{\mathbb{T}^{2}}\Re\left[\frac{1+zs}{1-zs}\right]\Re\left[\frac{1+wt}{1-wt}\right]\, d\mu(s,t),\qquad(z,w)\in\mathbb{D}^{2},\label{eq:2.2}
\end{equation}
recovers the values of the Poisson integral of the measure $d\mu(1/s,1/t)$,
determining the measure $\mu$ itself (cf. \cite{Rudin}).

Fix $\nu\in\mathscr{P}_{\mathbb{T}}$ and consider its 1-D $\psi$-transform
$\psi_{\nu}$. The \emph{$\eta$-transform} 
\[
\eta_{\nu}(z)=\frac{\psi_{\nu}(z)}{1+\psi_{\nu}(z)},\qquad z\in\mathbb{C}\setminus\mathbb{T},
\]
of $\nu$ is a holomorphic function satisfying $\eta_{\nu}(0)=0$
and $|\eta_{\nu}(z)|\leq|z|$ for $z\in\mathbb{D}$. The symmetry
of $\psi_{\nu}$ shows that $\eta_{\nu}$ also satisfies the Schwarz
type reflection formula
\begin{equation}
\eta_{\nu}(z)=\overline{1/\eta_{\nu}(1/\overline{z})},\qquad z\in\mathbb{C}\setminus\mathbb{T},\label{eq:2.3}
\end{equation}
and hence $\eta_{\nu}(\infty)=\infty$ and $|\eta_{\nu}(z)|\geq|z|$
for $\left|z\right|>1$. 

We introduce two appropriate classes of measures for free harmonic
analysis. The set $\mathscr{P}_{\mathbb{T}}^{\times}$ is the collection
of measures $\mu\in\mathscr{P}_{\mathbb{T}}$ such that the mean 
\[
m(\mu)=\int_{\mathbb{T}}x\, d\mu(x)\neq0.
\]
Next, the set $\mathscr{P}_{\mathbb{T}^{2}}^{\times}$ consists of
all probability measures $\mu\in\mathscr{P}_{\mathbb{T}^{2}}$ such
that $\mu^{(1)},\mu^{(2)}\in\mathscr{P}_{\mathbb{T}}^{\times}$ and
the moment 
\[
m_{1,1}(\mu)=\int_{\mathbb{T}^{2}}st\, d\mu(s,t)\neq0.
\]

The value of the derivative $\eta_{\nu}^{\prime}(0)$ is equal to
the mean $m(\nu)$. If $\nu\in\mathscr{P}_{\mathbb{T}}^{\times}$,
the map $\eta_{\nu}$ will be conformal near the origin, as well as
nearby the point of infinity by the reflection symmetry (\ref{eq:2.3}).
Thus, there exists a small radius $r=r(\nu)\in(0,1)$ such that the
inverse function $\eta_{\nu}^{-1}$ of $\eta_{\nu}$ is defined in
the open set $D_{r}\cup\Delta_{r}$, where $D_{r}=\left\{ z\in\mathbb{C}:|z|<r\right\} $
and $\Delta_{r}=\left\{ z\in\mathbb{C}:|z|>1/r\right\} $. The function
$\eta_{\nu}^{-1}$ is holomorphic in its domain of definition $D_{r}\cup\Delta_{r}$,
having a simple zero at $z=0$ and a simple pole at $z=\infty$, and
it maps the disk $D_{r}$ into $\mathbb{D}$ and the set $\Delta_{r}$
into $\Delta_{1}=\left\{ z\in\mathbb{C}:|z|>1\right\} $. In addition,
we have the symmetry 
\begin{equation}
\eta_{\nu}^{-1}(z)=\overline{1/\eta_{\nu}^{-1}(1/\overline{z})},\qquad z\in D_{r}\cup\Delta_{r}.\label{eq:2.4}
\end{equation}
Finally, the map $\eta_{\nu}^{-1}$ determines the measure $\nu$
uniquely, as one can recover the transform $\psi_{\nu}$ from $\eta_{\nu}^{-1}$
by inverting it. Also, we notice that $\psi_{\nu}$ is also invertible
near the origin and the point of infinity, with $\psi_{\nu}^{-1}(z)=\eta_{\nu}^{-1}(z/(1+z))$.

The inverse $\eta$-transform plays a role in free probability. Namely,
Voiculescu proved in \cite{VoiMulFree} that for any measures $\nu_{1},\nu_{2}\in\mathscr{P}_{\mathbb{T}}^{\times}$,
their (multiplicative) free convolution $\nu_{1}\boxtimes\nu_{2}$
also belongs to the class $\mathscr{P}_{\mathbb{T}}^{\times}$ and
the identity 
\[
z\eta_{\nu_{1}\boxtimes\nu_{2}}^{-1}(z)=\eta_{\nu_{1}}^{-1}(z)\eta_{\nu_{2}}^{-1}(z)
\]
holds in any common domain of definition for the three inverse $\eta$-transforms.

Given a measure $\mu\in\mathscr{P}_{\mathbb{T}^{2}}$, we also consider
the following integral transform
\[
H_{\mu}(z,w)=\int_{\mathbb{T}^{2}}\frac{1}{(1-zs)(1-wt)}\, d\mu(s,t),\qquad(z,w)\in(\mathbb{C}\setminus\mathbb{T})^{2},
\]
which is a holomorphic map satisfying the relationship 
\begin{equation}
H_{\mu}(z,w)=\psi_{\mu}(z,w)+\psi_{\mu^{(1)}}(z)+\psi_{\mu^{(2)}}(w)+1.\label{eq:2.5}
\end{equation}

\subsection{Bi-free partial $S$- and $\Sigma$-transforms}

Let $\mu$ be a measure in $\mathscr{P}_{\mathbb{T}^{2}}^{\times}$.
Recall from \cite{part3} that the bi-free partial $S$-transform
of $\mu$ is 
\[
S_{\mu}(z,w)=\frac{1+z}{z}\frac{1+w}{w}\left[1-\frac{1+z+w}{H_{\mu}\left(\psi_{\mu^{(1)}}^{-1}(z),\psi_{\mu^{(2)}}^{-1}(w)\right)}\right].
\]
We introduce the following function, called the \emph{bi-free partial
$\Sigma$-transform} (or, just $\Sigma$-transform for short), via
a change of variables: 
\[
\Sigma_{\mu}(z,w)=S_{\mu}\left(\frac{z}{1-z},\frac{w}{1-w}\right)=\frac{\psi_{\mu}\left(\eta_{\mu^{(1)}}^{-1}(z),\eta_{\mu^{(2)}}^{-1}(w)\right)}{zwH_{\mu}\left(\eta_{\mu^{(1)}}^{-1}(z),\eta_{\mu^{(2)}}^{-1}(w)\right)}.
\]

The domain of definition for $\Sigma_{\mu}$ is set to be an open
Reinhardt domain 
\[
\Omega_{r}=(D_{r}\cup\Delta_{r})\times(D_{r}\cup\Delta_{r}),
\]
where the number $r\in(0,1)$ is chosen small enough so that at least
the two inverses $\eta_{\mu^{(j)}}^{-1}$ ($j=1,2$) are defined in
$D_{r}\cup\Delta_{r}$. Furthermore, shrinking the number $r$ if
necessary, we can find a set $\Omega_{r}$ so that $\Sigma_{\mu}$
is a holomorphic function in $\Omega_{r}$. Indeed, for $(z,w)$ in
the bidisk connected component $D_{r}\times D_{r}$ of $\Omega_{r}$,
we may view the function $\Sigma_{\mu}$ as a quotient $f/g$ of holomorphic
maps, where 
\[
f(z,w)=\frac{\eta_{\mu^{(1)}}^{-1}(z)}{z}\frac{\eta_{\mu^{(2)}}^{-1}(w)}{w}\int_{\mathbb{T}^{2}}\frac{st}{(1-\eta_{\mu^{(1)}}^{-1}(z)s)(1-\eta_{\mu^{(2)}}^{-1}(w)t)}\, d\mu(s,t)
\]
and $g(z,w)=H_{\mu}\left(\eta_{\mu^{(1)}}^{-1}(z),\eta_{\mu^{(2)}}^{-1}(w)\right)$.
Since $\lim_{(z,w)\rightarrow(0,0)}g(z,w)=1$, there exists a small
number $r$ such that the denominator $g$ is zero-free in the corresponding
bidisk $D_{r}\times D_{r}$ and hence $\Sigma_{\mu}$ is well-defined
and holomorphic in that bidisk. As for $(z,w)$ in the unbounded component
$D_{r}\times\Delta_{r}$, we again write $\Sigma_{\mu}=F/G$ as a
quotient of holomorphic maps, only this time the denominator 
\begin{eqnarray*}
G(z,w) & = & wH_{\mu}\left(\eta_{\mu^{(1)}}^{-1}(z),\eta_{\mu^{(2)}}^{-1}(w)\right)\\
 & = & w\overline{\eta_{\mu^{(2)}}^{-1}(1/\overline{w})}\cdot\int_{\mathbb{T}^{2}}\frac{1}{(1-\eta_{\mu^{(1)}}^{-1}(z)s)(\overline{\eta_{\mu^{(2)}}^{-1}(1/\overline{w})}-t)}\, d\mu(s,t)\\
 &  & \rightarrow\overline{1/m(\mu^{(2)})}\cdot-\overline{m(\mu^{(2)})}=-1
\end{eqnarray*}
as $z\rightarrow0$ and $|w|\rightarrow\infty$. (Here the assumption
$m(\mu^{(2)})\neq0$ and the symmetry (\ref{eq:2.4}) are used to
evaluate the last limit.) So, the existence of the radius $r$ is
also guaranteed in this case. One can treat the other two components
$\Delta_{r}\times D_{r}$ and $\Delta_{r}\times\Delta_{r}$ in the
same way and conclude that $\Sigma_{\mu}$ is well-defined (and hence
holomorphic) in $\Omega_{r}$ for a suitable $r$. 

In addition, it is easy to see that 
\begin{equation}
\Sigma_{\mu}(z,w)=1/\overline{\Sigma_{\mu}(1/\overline{z},1/\overline{w})},\qquad(z,w)\in\Omega_{r},\label{eq:2.6}
\end{equation}
with the value $\Sigma_{\mu}(0,0)=m_{1,1}(\mu)/[m(\mu^{(1)})m(\mu^{(2)})]$
and $\lim_{z\rightarrow0,|w|\rightarrow\infty}\Sigma_{\mu}(z,w)=1$. 
\begin{rem*}
If we only assume $m(\mu^{(1)})\neq0\neq m(\mu^{(2)})$, then the
map $\Sigma_{\mu}$ is still defined in $\Omega_{r}$, except possibly
in the unbounded component $\Delta_{r}\times\Delta_{r}$. For such
$\mu$, it is clear that $m_{1,1}(\mu)\neq0$ (so that $\mu\in\mathscr{P}_{\mathbb{T}^{2}}^{\times}$)
if and only if $\Sigma_{\mu}(0,0)\neq0$. 
\end{rem*}
As shown in \cite{part3}, the bi-free convolution $\mu_{1}\boxtimes\boxtimes\mu_{2}$
for any $\mu_{1},\mu_{2}\in\mathscr{P}_{\mathbb{T}^{2}}^{\times}$
is again in the class $\mathscr{P}_{\mathbb{T}^{2}}^{\times}$, and
the multiplicative property of the bi-free partial $S$-transform
yields 
\[
\Sigma_{\mu_{1}\boxtimes\boxtimes\mu_{2}}(z,w)=\Sigma_{\mu_{1}}(z,w)\,\Sigma_{\mu_{2}}(z,w)
\]
for $(z,w)$ in the bidisk component $D_{r}\times D_{r}$ of a product
set $\Omega_{r}$ on which all the involved $\Sigma$-transforms are
defined. 

The knowledge of the function $\Sigma_{\mu}$ alone is insufficient
to determine the measure $\mu$. For example, it was observed in \cite{Skoufranis}
that the transform $\Sigma_{\delta_{\lambda}}$ is constantly one
(in any $\Omega_{r}$), so that any measure $\mu$ and its marginal
rotation $\mu\boxtimes\boxtimes\delta_{\lambda}$ always share the
same $\Sigma$-transform. However, if the two marginals $\mu^{(1)}$,
$\mu^{(2)}$ are specified in advance, then the map $\mu\mapsto\Sigma_{\mu}$
will be injective. 
\begin{prop}
Let $\mu$ and $\nu$ be two measures in $\mathscr{P}_{\mathbb{T}^{2}}^{\times}$
such that $\mu^{(j)}=\nu^{(j)}$ for $j=1,2$. Then we have $\Sigma_{\mu}=\Sigma_{\nu}$
on a product set $\Omega_{r}$  if and only if $\mu=\nu$. \end{prop}
\begin{proof}
Only the necessity statement requires a proof. Assume the equalities
$\mu^{(j)}=\nu^{(j)}$ ($j=1,2$) and $\Sigma_{\mu}=\Sigma_{\nu}$
on some $\Omega_{r}=(D_{r}\cup\Delta_{r})\times(D_{r}\cup\Delta_{r})$.
Since any $\eta$-transform is a contraction toward zero in $\mathbb{D}$
and an expansion toward $\infty$ in $\Delta_{1}$, one has $\eta_{\mu^{(j)}}(D_{r/2})\subset D_{r}$
and $\eta_{\mu^{(j)}}(\Delta_{r/2})\subset\Delta_{r}$ for $j=1,2$.
Then we obtain 
\begin{eqnarray*}
\psi_{\mu}(z,w) & = & \left[\psi_{\mu^{(1)}}(z)+\psi_{\mu^{(2)}}(w)+1\right]\frac{\eta_{\mu^{(1)}}(z)\eta_{\mu^{(2)}}(w)\Sigma_{\mu}\left(\eta_{\mu^{(1)}}(z),\eta_{\mu^{(2)}}(w)\right)}{1-\eta_{\mu^{(1)}}(z)\eta_{\mu^{(2)}}(w)\Sigma_{\mu}\left(\eta_{\mu^{(1)}}(z),\eta_{\mu^{(2)}}(w)\right)}\\
 & = & \psi_{\nu}(z,w)
\end{eqnarray*}
for $(z,w)$ in the smaller product set $\Omega_{r/2}$. Since $\Omega_{r/2}$
is open in $\mathbb{C}^{2}$ and has a nonempty intersection with
each of the four connected components of $(\mathbb{C}\setminus\mathbb{T})^{2}$,
we conclude by analyticity that $\psi_{\mu}=\psi_{\nu}$ in $(\mathbb{C}\setminus\mathbb{T})^{2}$.
As a result, the Poisson integral formula (\ref{eq:2.2}) yields $\mu=\nu$. 
\end{proof}
The preceding result also holds in the larger class of probability measures with nonzero marginal means. However, such a general result is not needed in this paper. 

The proof of Proposition 2.5 shows a procedure of recovering the measure $\mu$
from its marginal laws and the map $\Sigma_{\mu}$. Moreover, we see
that the case $\Sigma_{\mu}=1$ corresponds to the situation $\psi_{\mu}=\psi_{\mu^{(1)}}\cdot\psi_{\mu^{(2)}}=\psi_{\mu^{(1)}\otimes\mu^{(2)}}$.
In particular, one obtains the following result.
\begin{cor}
For $\mu\in\mathscr{P}_{\mathbb{T}^{2}}^{\times}$, the transform
$\Sigma_{\mu}$ is constantly one in a product set $\Omega_{r}$ if
and only if the measure $\mu$ is the product measure of its own marginal
laws. 
\end{cor}
We next introduce appropriate transforms to treat the opposite bi-free
convolution. 
\begin{defn}
For any $\mu\in\mathscr{P}_{\mathbb{T}^{2}}^{\times}$, we define
its \emph{opposite bi-free partial $S$-transform} $S_{\mu}^{\text{op}}$
by
\[
S_{\mu}^{\text{op}}(z,w)=\frac{w(z+1)}{z(w+1)}\left[1+\frac{z-w}{H_{\mu}\left(\psi_{\mu^{(1)}}^{-1}(z),\psi_{\mu^{(2)}}^{-1}(w)\right)-z-1}\right].
\]
The \emph{opposite bi-free partial $\Sigma$-transform }of $\mu$
is defined as 
\[
\Sigma_{\mu}^{\text{op}}(z,w)=S_{\mu}^{\text{op}}\left(\frac{z}{1-z},\frac{w}{1-w}\right).
\]

\end{defn}
Below is a list of important properties of these opposite transforms,
in which the multiplicative property (3) is derived by following Voiculescu's
original arguments in \cite{part3}. 
\begin{prop}
We have:
\begin{enumerate}
\item To each $\mu\in\mathscr{P}_{\mathbb{T}^{2}}^{\times}$, there exists
$r=r(\mu)\in(0,1)$ such that the opposite transforms $\Sigma_{\mu}^{\text{\emph{op}}}$
and $S_{\mu}^{\text{\emph{op}}}$ are well-defined and holomorphic
in the bidisk $D_{r}\times D_{r}$. 
\item The formula 
\[
\Sigma_{\mu^{*}}^{\text{\emph{op}}}(z,w)=\Sigma_{\mu}(z,1/w)
\]
holds for $(z,w)\in D_{r}\times D_{r}$ and $w\neq0$, where $\mu^{*}$
is the coordinate-reflection of $\mu$. 
\item If $\mu_{1},\mu_{2}\in\mathscr{P}_{\mathbb{T}^{2}}^{\times}$ then
the multiplicative identity 
\[
\Sigma_{\mu_{1}\boxtimes\boxtimes^{\text{\emph{op}}}\mu_{2}}^{\text{\emph{op}}}(z,w)=\Sigma_{\mu_{1}}^{\text{\emph{op}}}(z,w)\,\Sigma_{\mu_{2}}^{\text{\emph{op}}}(z,w)
\]
holds for $(z,w)$ in a small bidisk centered at the point $(0,0)$.
\end{enumerate}
\end{prop}
\begin{proof}
The proof of (1) is straightforward and is similar to the case of
$\Sigma_{\mu}$. Indeed, we use (\ref{eq:2.5}) to re-write the map
$\Sigma_{\mu}^{\text{op}}$ as the quotient 
\[
\Sigma_{\mu}^{\text{op}}(z,w)=\frac{\psi_{\mu}\left(\eta_{\mu^{(1)}}^{-1}(z),\eta_{\mu^{(2)}}^{-1}(w)\right)/z+1/(1-z)}{\psi_{\mu}\left(\eta_{\mu^{(1)}}^{-1}(z),\eta_{\mu^{(2)}}^{-1}(w)\right)/w+1/(1-w)}
\]
of holomorphic maps in a neighborhood of $(0,0)$. Since the denominator
has limit $1$ at $(0,0)$, such a radius $r$ can always be found.
The case of $S_{\mu}^{\text{op}}$ follows from the substitution $z\mapsto z/(1+z) $, $w\mapsto w/(1+w)$. 

For (2), we first note that 
\[
\psi_{[\mu^{*}]^{(2)}}(w)=\overline{\psi_{\mu^{(2)}}(\overline{w})},\quad\eta_{[\mu^{*}]^{(2)}}(w)=\overline{\eta_{\mu^{(2)}}(\overline{w})}
\]
and $\psi_{[\mu^{*}]^{(1)}}(z)=\psi_{\mu^{(1)}}(z)$ for $(z,w)\in(\mathbb{C}\setminus\mathbb{T})^{2}$.
Therefore, for $(z,w)\in D_{r}\times D_{r}$ and $w\neq0$, we compute
\begin{eqnarray*}
\psi_{\mu^{*}}\left(\eta_{[\mu^{*}]^{(1)}}^{-1}(z),\eta_{[\mu^{*}]^{(2)}}^{-1}(w)\right) & = & -\psi_{\mu}\left(\eta_{\mu^{(1)}}^{-1}(z),1/\overline{\eta_{\mu^{(2)}}^{-1}(\overline{w})}\right)-\psi_{\mu^{(1)}}\left(\eta_{\mu^{(1)}}^{-1}(z)\right)\\
 & = & -\psi_{\mu}\left(\eta_{\mu^{(1)}}^{-1}(z),\eta_{\mu^{(2)}}^{-1}(1/w)\right)-z/(1-z)
\end{eqnarray*}
and 
\begin{eqnarray*}
\Sigma_{\mu^{*}}^{\text{op}}(z,w) & = & \frac{1}{z(1/w)}\left[\frac{\psi_{\mu^{*}}\left(\eta_{[\mu^{*}]^{(1)}}^{-1}(z),\eta_{[\mu^{*}]^{(2)}}^{-1}(w)\right)+z/(1-z)}{\psi_{\mu^{*}}\left(\eta_{[\mu^{*}]^{(1)}}^{-1}(z),\eta_{[\mu^{*}]^{(2)}}^{-1}(w)\right)+w/(1-w)}\right]\\
 & = & \frac{1}{z(1/w)}\left[\frac{\psi_{\mu}\left(\eta_{\mu^{(1)}}^{-1}(z),\eta_{\mu^{(2)}}^{-1}(1/w)\right)}{\psi_{\mu}\left(\eta_{\mu^{(1)}}^{-1}(z),\eta_{\mu^{(2)}}^{-1}(1/w)\right)+z/(1-z)+1/(w-1)+1}\right]\\
 & = & \Sigma_{\mu}(z,1/w),
\end{eqnarray*}
as desired.

We now prove (3). For reader's convenience of cross-referring to \cite{part3},
we present this proof using the same notations in Voiculescu's original
arguments. Thus, let $(a_{1},b_{1})$ and $(a_{2},b_{2})$ be two
two-faced pairs of commuting unitaries in a $C^{*}$-probability space
$(\mathcal{A},\varphi)$. Assume that $(a_{1},b_{1})$ and $(a_{2},b_{2})$
are bi-free and that $\varphi(a_{k})\neq0$ and $\varphi(b_{k})\neq0$
for $k=1,2$. Then we aim to prove the identity 
\[
S_{(a_{1}a_{2},b_{2}b_{1})}^{\text{op}}(z,w)=S_{(a_{1},b_{1})}^{\text{op}}(z,w)\, S_{(a_{2},b_{2})}^{\text{op}}(z,w)
\]
for $(z,w)$ near the point $(0,0)$. Here, by the spectral theorem,
the opposite $S$-transform 
\[
S_{(a,b)}^{\text{op}}(z,w)=\frac{w(z+1)}{z(w+1)}\left[1+\frac{z-w}{H_{(a,b)}\left(\psi_{a}^{-1}(z),\psi_{b}^{-1}(w)\right)-z-1}\right]
\]
for a pair $(a,b)$ of commuting unitaries is interpreted as an absolutely
convergent power series near $(0,0)$, where 
\[
H_{(a,b)}(t,s)=\varphi\left([1-ta]^{-1}[1-sb]^{-1}\right).
\]

Recall from \cite{part3} that the functions $h_{a}(t)=\varphi\left([1-ta]^{-1}\right)$,
$h_{b}(s)=\varphi\left([1-sb]^{-1}\right)$ and the centered resolvents
\[
a(t)=(1-ta)^{-1}-h_{a}(t)1,\quad b(s)=(1-sb)^{-1}-h_{b}(s)1,
\]
are well-defined if the complex numbers $t$ and $s$ come sufficiently
close to zero. 

Using Lemma 2.4 in \cite{part3}, we know that if two nonzero complex
numbers $z$ and $w$ are sufficiently to zero, then the following
complex numbers 
\[
\begin{cases}
t_{k}=\psi_{a_{k}}^{-1}(z),\; s_{k}=\psi_{b_{k}}^{-1}(w), & k=1,2,\\
t=\psi_{a_{1}a_{2}}^{-1}(z),\; s=\psi_{b_{2}b_{1}}^{-1}(w),\\
\rho=1/(z^{2}+z),\;\sigma=1/(w^{2}+w),
\end{cases}
\]
are also nonzero and satisfy
\[
\begin{cases}
\left\Vert \rho a_{1}(t_{1})a_{2}(t_{2})\right\Vert <1,\;\left\Vert \sigma b_{2}(s_{2})b_{1}(s_{1})\right\Vert <1,\\
h_{a_{1}a_{2}}(t)=h_{a_{1}}(t_{1})=h_{a_{2}}(t_{2})=z+1\notin\left\{ 0,1\right\} ,\\
h_{b_{2}b_{1}}(s)=h_{b_{1}}(s_{1})=h_{b_{2}}(s_{2})=w+1\notin\left\{ 0,1\right\} ,\\
(z+1)(1-ta_{1}a_{2})^{-1}=[a_{2}(t_{2})+h_{a_{2}}(t_{2})1][1-\rho a_{1}(t_{1})a_{2}(t_{2})]^{-1}[a_{1}(t_{1})+h_{a_{1}}(t_{1})1],\\
(w+1)(1-sb_{2}b_{1})^{-1}=[b_{1}(s_{1})+h_{b_{1}}(s_{1})1][1-\sigma b_{2}(s_{2})b_{1}(s_{1})]^{-1}[b_{2}(s_{2})+h_{b_{2}}(s_{2})1].
\end{cases}
\]
(These results are originally due to Haagerup in \cite{Haagerup}.)
In the sequel we will confine ourselves to the case of such $z$ and
$w$. 

As in \cite{part3}, we start with a calculation of mixed moments
of resolvents. For any $j,k\geq0$, consider the product
\begin{eqnarray*}
* & = & \left\{ [a_{2}(t_{2})+h_{a_{2}}(t_{2})1][\rho a_{1}(t_{1})a_{2}(t_{2})]^{j}[a_{1}(t_{1})+h_{a_{1}}(t_{1})1]\right.\\
 &  & \left.[b_{1}(s_{1})+h_{b_{1}}(s_{1})1][\sigma b_{2}(s_{2})b_{1}(s_{1})]^{k}[b_{2}(s_{2})+h_{b_{2}}(s_{2})1]\right\} .
\end{eqnarray*}
By Lemma 2.2, only the following products of random variables can
make a nonzero contribution to the mixed moment $\varphi(*)$; namely,
\[
\begin{cases}
*_{1}=a_{2}(t_{2})[\rho a_{1}(t_{1})a_{2}(t_{2})]^{j}a_{1}(t_{1})h_{b_{1}}(s_{1})[\sigma b_{2}(s_{2})b_{1}(s_{1})]^{k}h_{b_{2}}(s_{2}), & k=j+1;\\
*_{2}=h_{a_{2}}(t_{2})[\rho a_{1}(t_{1})a_{2}(t_{2})]^{j}a_{1}(t_{1})b_{1}(s_{1})[\sigma b_{2}(s_{2})b_{1}(s_{1})]^{k}h_{b_{2}}(s_{2}), & k=j;\\
*_{3}=a_{2}(t_{2})[\rho a_{1}(t_{1})a_{2}(t_{2})]^{j}h_{a_{1}}(t_{1})h_{b_{1}}(s_{1})[\sigma b_{2}(s_{2})b_{1}(s_{1})]^{k}b_{2}(s_{2}), & k=j;\\
*_{4}=h_{a_{2}}(t_{2})[\rho a_{1}(t_{1})a_{2}(t_{2})]^{j}h_{a_{1}}(t_{1})b_{1}(s_{1})[\sigma b_{2}(s_{2})b_{1}(s_{1})]^{k}b_{2}(s_{2}), & k=j-1,\: j\geq1;\\
*_{5}=h_{a_{2}}(t_{2})h_{a_{1}}(t_{1})h_{b_{1}}(s_{1})h_{b_{2}}(s_{2})1=(z+1)^{2}(w+1)^{2}1.
\end{cases}
\]
(The product $*_{5}$ corresponds to the case of $j=0=k$.) So, we
have the expectations
\[
\begin{cases}
\varphi(*_{1})=(\rho\sigma x_{1}x_{2})^{j}\frac{w+1}{w}x_{1}x_{2}, & j\geq0;\\
\varphi(*_{2})=(\rho\sigma x_{1}x_{2})^{j}(z+1)(w+1)x_{1}, & j\geq0;\\
\varphi(*_{3})=(\rho\sigma x_{1}x_{2})^{j}(z+1)(w+1)x_{2}, & j\geq0;\\
\varphi(*_{4})=(\rho\sigma x_{1}x_{2})^{j}(z+1)^{2}w(w+1), & j\geq1;\\
\varphi(*_{5})=(z+1)^{2}(w+1)^{2},
\end{cases}
\]
where
\[
x_{k}=\varphi(a_{k}(t_{k})b_{k}(s_{k}))=H_{(a_{k},b_{k})}(t_{k},s_{k})-zw-w-z-1,\quad k=1,2.
\]
We next compute
\begin{eqnarray*}
(z+1)(w+1)H_{(a_{1}a_{2},b_{2}b_{1})}(t,s) & = & \varphi\left([z+1][1-ta_{1}a_{2}]^{-1}[w+1][1-sb_{2}b_{1}]^{-1}\right)\\
 & = & \sum_{|j|+|k|\geq0}\varphi(*)\\
 & = & \varphi(*_{5})+\sum_{|j|+|k|\geq 1}\left\{ \varphi(*_{1})+\varphi(*_{2})+\varphi(*_{3})+\varphi(*_{4})\right\} \\
 & = & (z+1)^{2}(w+1)^{2}-(z+1)^{2}w(w+1)\\
 &  & +\frac{w+1}{w}(x_{1}+zw+w)(x_{2}+zw+w)\sum_{j=0}^{\infty}(\rho\sigma x_{1}x_{2})^{j}\\
 & = & (z+1)^{2}(w+1)\\
 &  & +\frac{(w+1)(x_{1}+zw+w)(x_{2}+zw+w)}{w(1-\rho\sigma x_{1}x_{2})}.
\end{eqnarray*}
For $z\neq w$, we introduce further the notations 
\[
F_{k}=\frac{H_{(a_{k},b_{k})}(t_{k},s_{k})-z-1}{z-w}=\frac{x_{k}+zw+w}{z-w}\qquad(k=1,2)
\]
and 
\[
F=\frac{H_{(a_{1}a_{2},b_{2}b_{1})}(t,s)-z-1}{z-w},
\]
so that we can re-cast the last result into 
\[
F=F_{1}F_{2}\,\frac{z-w}{w(z+1)(1-\rho\sigma x_{1}x_{2})}.
\]
Since the factor 
\begin{eqnarray*}
\frac{w(z+1)(1-\rho\sigma x_{1}x_{2})}{z-w} & = & \frac{w(z+1)}{z-w}\left[1-\frac{x_{1}x_{2}}{zw(z+1)(w+1)}\right]\\
 & = & \frac{z-w}{z(w+1)}\left[-F_{1}F_{2}+\frac{w(z+1)}{z-w}F_{1}\right.\\
 &  & \left.+\frac{w(z+1)w}{z-w}F_{2}+\frac{w(z+1)}{z-w}\right],
\end{eqnarray*}
we obtain 
\[
1+\frac{1}{F}=\frac{w(z+1)}{z(w+1)}\left[1+\frac{1}{F_{1}}\right]\left[1+\frac{1}{F_{2}}\right],
\]
which is precisely what we are set to prove in the beginning. 

The statement (3) now follows by choosing appropriate bi-free random
vectors according to the given distributions $\mu_{1}$ and $\mu_{2}$
(see Section 2.1). \end{proof}
\begin{cor}
Let $\mu_{1},\mu_{2}\in\mathscr{P}_{\mathbb{T}^{2}}^{\times}$ and
let $\Omega_{r}$ be a common domain of definition for the transforms
$\Sigma_{\mu_{1}}$, $\Sigma_{\mu_{2}}$, and $\Sigma_{\mu_{1}\boxtimes\boxtimes\mu_{2}}$.
Then we have the multiplicative identity 
\[
\Sigma_{\mu_{1}\boxtimes\boxtimes\mu_{2}}(z,w)=\Sigma_{\mu_{1}}(z,w)\,\Sigma_{\mu_{2}}(z,w),\qquad(z,w)\in\Omega_{r}.
\]
\end{cor}
\begin{proof}
By (\ref{eq:2.6}), it suffices to prove this identity on the unbounded
component $D_{r}\times\Delta_{r}$. Indeed, for $(z,w)\in D_{r}\times\Delta_{r}$,
we have 
\begin{eqnarray*}
\Sigma_{\mu_{1}\boxtimes\boxtimes\mu_{2}}(z,w) & = & \Sigma_{(\mu_{1}\boxtimes\boxtimes\mu_{2})^{*}}^{\text{op}}(z,1/w)\\
 & = & \Sigma_{\mu_{1}^{*}\boxtimes\boxtimes^{\text{op}}\mu_{2}^{*}}^{\text{op}}(z,1/w)\\
 & = & \Sigma_{\mu_{1}^{*}}^{\text{op}}(z,1/w)\,\Sigma_{\mu_{2}^{*}}^{\text{op}}(z,1/w)\\
 & = & \Sigma_{\mu_{1}}(z,w)\,\Sigma_{\mu_{2}}(z,w).
\end{eqnarray*}
\end{proof}

By 1-D free harmonic analysis \cite{BerVoi}, a sequence $\{\nu_{n}\}_{n=1}^{\infty}\subset\mathscr{P}_{\mathbb{T}}^{\times}$
converges weakly to $\nu\in\mathscr{P}_{\mathbb{T}}^{\times}$ if
and only if there exists $r\in(0,1)$ so that all $\eta_{\nu_{n}}^{-1}$
and $\eta_{\nu}^{-1}$ are defined in the open set $D_{r}\cup\Delta_{r}$
and $\eta_{\nu_{n}}^{-1}\rightarrow\eta_{\nu}^{-1}$ locally uniformly
in $D_{r}\cup\Delta_{r}$, i.e., uniform convergence over compact
subsets of $D_{r}\cup\Delta_{r}$. On the other hand, the weak convergence
$\nu_{n}\Rightarrow\nu$ is also equivalent to the local uniform convergence
$\psi_{\nu_{n}}\rightarrow\psi_{\nu}$ in $\mathbb{C}\setminus\mathbb{T}$.
We end this section with a continuity theorem for the $\Sigma$-transform. 
\begin{thm}
Let $\{\mu_{n}\}_{n=1}^{\infty}$ be a sequence in $\mathscr{P}_{\mathbb{T}^{2}}^{\times}$.
The sequence $\{\mu_{n}\}_{n=1}^{\infty}$ converges weakly to a probability
measure in $\mathscr{P}_{\mathbb{T}^{2}}^{\times}$ if and only if
(i) both $\{\mu_{n}^{(1)}\}_{n=1}^{\infty}$ and $\{\mu_{n}^{(2)}\}_{n=1}^{\infty}$
are weakly convergent in $\mathscr{P}_{\mathbb{T}}^{\times}$ and
(ii) there exists a domain of definition $\Omega_{r}=(D_{r}\cup\Delta_{r})\times(D_{r}\cup\Delta_{r})$
such that the pointwise limit $\Sigma(z,w)=$$\lim_{n\rightarrow\infty}\Sigma_{\mu_{n}}(z,w)$
exists for all $(z,w)\in\Omega_{r}$ and $\Sigma(0,0)\neq0$. Moreover,
if $\mu_{n}\Rightarrow\mu$ then we have $\Sigma=\Sigma_{\mu}$ on
$\Omega_{r}$. \end{thm}
\begin{proof}
Suppose we have $\mu_{n}\Rightarrow\mu$ with $\mu\in\mathscr{P}_{\mathbb{T}^{2}}^{\times}$.
The weak convergence $\mu_{n}\Rightarrow\mu$ already implies the
pointwise convergences $H_{\mu_{n}}\rightarrow H_{\mu}$, $\psi_{\mu_{n}}\rightarrow\psi_{\mu}$
in $(\mathbb{C}\setminus\mathbb{T})^{2}$. The marginal weak convergence
is a consequence of the continuity of the projection $\pi_{j}$. Denote by $\Omega_{r}$ the domain of definition for the limiting
transform $\Sigma_{\mu}$. By choosing a smaller $r$ if necessary, we may also assume that all $\eta_{\mu_{n}^{(j)}}^{-1}$ are defined on the projected image $\pi_{j}\left(\Omega_{r}\right)$ for $j=1,2$. We
claim that $\Omega_{r}$ is the universal domain we are looking for
and the pointwise convergence $\Sigma_{\mu_{n}}\rightarrow\Sigma_{\mu}$
holds in $\Omega_{r}$. 

Toward this end, we first establish the local uniform convergence
for $H_{\mu_{n}}$ and $\psi_{\mu_{n}}$. Fix an arbitrary $r_{1}\in(0,1)$
and denoting $\Omega_{r_{1}}=(D_{r_{1}}\cup\Delta_{r_{1}})\times(D_{r_{1}}\cup\Delta_{r_{1}})$,
we observe that 
\[
\left|\partial_{z}H_{\mu_{n}}(z,w)\right|=\left|\int_{\mathbb{T}^{2}}\frac{s\, d\mu_{n}(s,t)}{(1-zs)^{2}(1-wt)}\right|<\frac{1}{(1-r_{1})^{3}},\qquad(z,w)\in\Omega_{r_{1}},\quad n\geq1.
\]
Likewise, the complex partial derivative $\partial_{w}H_{\mu_{n}}$
is uniformly bounded by $(1-r_{1})^{-3}$ in the product set $\Omega_{r_{1}}$.
Hence the operator norm of the complex differential $DH_{\mu_{n}}(z,w)$
is uniformly bounded by $\sqrt{2}(1-r_{1})^{-3}$ for all $n$ and
$(z,w)\in\Omega_{r_{1}}$. After applying the mean value theorem for
$\mathbb{R}^{4}\simeq\mathbb{C}^{2}$ to $H_{\mu_{n}}$, we conclude
that the sequence $\{H_{\mu_{n}}\}_{n=1}^{\infty}$ is equicontinuous
in $\Omega_{r_{1}}$.

On the other hand, we have the uniform bound $\left|H_{\mu_{n}}\right|\leq(1-r_{1})^{-2}$
in $\Omega_{r_{1}}$ for all $n\geq1$. Together with the equicontinuity, this shows that the family $\{H_{\mu_{n}}\}_{n=1}^{\infty}$
of holomorphic maps is in fact pre-compact in the topology of local
uniform convergence. Since the pointwise limit function $H_{\mu}$
is already holomorphic, any other holomorphic limit of $\{H_{\mu_{n}}\}_{n=1}^{\infty}$
must coincide with $H_{\mu}$ in $\Omega_{r_{1}}$. This proves the
local uniform convergence $H_{\mu_{n}}\rightarrow H_{\mu}$ in $\Omega_{r_{1}}$.
Since $(\mathbb{C}\setminus\mathbb{T})^{2}=\cup_{0<r_{1}<1}\Omega_{r_{1}}$,
this convergence actually holds locally uniformly in $(\mathbb{C}\setminus\mathbb{T})^{2}$.

The local uniform convergence $\psi_{\mu_{n}}\rightarrow\psi_{\mu}$
in $(\mathbb{C}\setminus\mathbb{T})^{2}$ follows from the reflection
symmetry $H_{\mu_{n}}(z,w)=\overline{\psi_{\mu_{n}}(1/\overline{z},1/\overline{w})}$.

We now argue that $\Sigma_{\mu_{n}}$ is well-defined on $\Omega_{r}$
for sufficiently large $n$. As in the beginning of this section,
we view $\Sigma_{\mu_{n}}$ as a quotient of holomorphic functions
and consider the denominators 
\[
g_{n}(z,w)=H_{\mu_{n}}\left(\eta_{\mu_{n}^{(1)}}^{-1}(z),\eta_{\mu_{n}^{(2)}}^{-1}(w)\right)
\]
and 
\[
G_{n}(z,w)=w\overline{\eta_{\mu_{n}^{(2)}}^{-1}(1/\overline{w})}\cdot\int_{\mathbb{T}^{2}}\frac{1}{(1-\eta_{\mu_{n}^{(1)}}^{-1}(z)s)(\overline{\eta_{\mu_{n}^{(2)}}^{-1}(1/\overline{w})}-t)}\, d\mu_{n}(s,t)
\]
in the components $ $$D_{r}\times D_{r}$ and $ $$D_{r}\times\Delta_{r}$,
respectively. Note that the local uniform convergence $H_{\mu_{n}}\rightarrow H_{\mu}$
in $(\mathbb{C}\setminus\mathbb{T})^{2}$ and that of $\eta_{\mu_{n}^{(j)}}^{-1}\rightarrow\eta_{\mu^{(j)}}^{-1}$
near the origin imply that the two convergences \[\lim_{n\rightarrow\infty}g_{n}(z,w)=H_{\mu}\left(\eta_{\mu^{(1)}}^{-1}(z),\eta_{\mu^{(2)}}^{-1}(w)\right)\]
and \[\lim_{n\rightarrow\infty}G_{n}(z,w)=wH_{\mu}\left(\eta_{\mu^{(1)}}^{-1}(z),\eta_{\mu^{(2)}}^{-1}(w)\right)\]
are uniform for $(z,w)$ in $D_{r}\times D_{r}$ and in $ $$D_{r}\times\Delta_{r}$. In particular, the denominators $g_{n}$, $G_{n}$ will
not vanish in these components when $n$ is sufficiently large, showing
that $\Sigma_{\mu_{n}}$ is well-defined there. The case of the other
two components $ $$\Delta_{r}\times\Delta_{r}$, $\Delta_{r}\times D_{r}$
is proved in the same way. 

If $(z,w)\in\Omega_{r}$, the pointwise limit 
\[
\lim_{n\rightarrow\infty}\Sigma_{\mu_{n}}(z,w)=\Sigma_{\mu}(z,w)
\]
follows from the local uniform convergences $H_{\mu_{n}}\rightarrow H_{\mu}$,
$\psi_{\mu_{n}}\rightarrow\psi_{\mu}$, and $\eta_{\mu_{n}^{(j)}}^{-1}\rightarrow\eta_{\mu^{(j)}}^{-1}$. 

Conversely, assume the weak convergence of the two marginals $\{\mu_{n}^{(1)}\}_{n=1}^{\infty}$
and $\{\mu_{n}^{(2)}\}_{n=1}^{\infty}$ in $\mathscr{P}_{\mathbb{T}}^{\times}$
and the pointwise convergence $\Sigma_{\mu_{n}}\rightarrow\Sigma$ in some universal domain of definition $\Omega_{r}$. Let $\mu$ and
$\mu^{\prime}$ be any two probability weak limit points for the sequence
$\{\mu_{n}\}_{n=1}^{\infty}$. Clearly, the measures $\mu$ and $\mu^{\prime}$
must have the same marginal laws, implying that their marginal means
are nonzero and $\Sigma_{\mu}$, $\Sigma_{\mu^{\prime}}$ are at least
defined in a neighborhood of $(0,0)$. The first part of the proof
shows that $\Sigma_{\mu}(0,0)=\Sigma_{\mu^{\prime}}(0,0)=\Sigma(0,0)\neq 0$.
Thus, the measures $\mu$, $\mu^{\prime}$ actually belong to the class $\mathscr{P}_{\mathbb{T}^{2}}^{\times}$ and hence both $\Sigma_{\mu}$ and $\Sigma_{\mu^{\prime}}$ are defined in $\Omega_{r}$, with $\Sigma_{\mu}=\Sigma=\Sigma_{\mu^{\prime}}$.
Proposition 2.5 then implies $\mu=\mu^{\prime}$, whence $\mu_{n}\Rightarrow\mu$. 
\end{proof}

\section{limit theorems}

We consider an infinite array $\{\mu_{nk}\}_{n,k}=\{\mu_{nk}:n\geq1,1\leq k\leq k_{n}\}$
in $\mathscr{P}_{\mathbb{T}^{2}}$ satisfying (\ref{eq:1.1}) and
a sequence $\lambda_{n}\in\mathbb{T}^{2}$. The condition (\ref{eq:1.1})
implies that the marginal laws are infinitesimal over $\mathbb{T}$,
that is, 
\[
\lim_{n\rightarrow\infty}\max_{1\leq k\leq k_{n}}\mu_{nk}^{(j)}(\{x\in\mathbb{T}:|x-1|\geq\varepsilon\})=0,\qquad\varepsilon>0,\; j=1,2.
\]
The goal of this section is to study the weak convergence of the measures
\[
\delta_{\lambda_{n}}\boxtimes\boxtimes\mu_{n1}\boxtimes\boxtimes\mu_{n2}\boxtimes\boxtimes\cdots\boxtimes\boxtimes\mu_{nk_{n}},\qquad n\geq1.
\]

Since infinitesimal measures fall into the class $\mathscr{P}_{\mathbb{T}^{2}}^{\times}$
eventually, we may assume further that the array $\{\mu_{nk}\}_{n,k}$
is already in the class $\mathscr{P}_{\mathbb{T}^{2}}^{\times}$,
so that the corresponding $\Sigma$-transforms are all well-defined,
giving us a ground to do analysis with them. 

Recall that a $\Sigma$-transform is invariant under arbitrary marginal
rotation. The key to proving our limit theorems lies in the fact that
there exists at least one such rotation allowing us to apply the free
limit theorems from \cite{BerJC} to the bi-free case. 

Thus, we follow \cite{BerJC} and introduce the complex numbers 
\[
b_{nk}^{(j)}=\exp\left(i\int_{|\arg x|<\varepsilon}\arg x\, d\mu_{nk}^{(j)}(x)\right),\qquad j=1,2,
\]
where the parameter $\varepsilon\in(0,1]$ will be arbitrary but fixed
in our discussions and $\arg x\in(-\pi,\pi]$ denotes the principal
value of the argument of $x$. Define a probability measure $\nu_{nk}=\delta_{(1/b_{nk}^{(1)},1/b_{nk}^{(2)})}\boxtimes\boxtimes\mu_{nk}$,
so that 
\begin{equation}
\nu_{nk}^{(j)}=\delta_{1/b_{nk}^{(j)}}\boxtimes\mu_{nk}^{(j)},\qquad j=1,2.\label{eq:3.1}
\end{equation}

One has $\lim_{n\rightarrow\infty}\max_{1\leq k\leq k_{n}}\left|b_{nk}^{(j)}-1\right|=0$
for $j=1,2$, and the resulting array $\{\nu_{nk}\}_{n,k}$ is also
infinitesimal over $\mathbb{T}^{2}$. Accordingly, we introduce the
holomorphic function
\[
h_{nk}^{(j)}(z)=\int_{\mathbb{T}}\frac{(1-z)(1-x)}{1-zx}\, d\nu_{nk}^{(j)}(x),\qquad z\in\mathbb{C}\setminus\mathbb{T}.
\]
These auxiliary functions satisfy the following properties (cf. \cite{BerJC}): 
\begin{enumerate}
\item[(\textbf{P1})] One has $\Re h_{nk}^{(j)}>0$ in $\mathbb{D}$; $\Re h_{nk}^{(j)}(z)=0$
for some $z\in\mathbb{D}$ if and only if $\nu_{nk}^{(j)}=\delta_{1}$.
\item[(\textbf{P2})] The symmetry 
\[
h_{nk}^{(j)}(z)=-\overline{h_{nk}^{(j)}(1/\overline{z})},\qquad z\in\mathbb{C}\setminus\mathbb{T},
\]
shows that the values of $h_{nk}^{(j)}$ on $\Delta_{1}$ are completely
determined by that of $h_{nk}^{(j)}$ on $\mathbb{D}$. 
\item[(\textbf{P3})] To each $r\in(0,1)$ and the cutoff constant $\varepsilon$,
there exists a constant $M=M(r,\varepsilon)>0$ such that 
\[
|\Im h_{nk}^{(j)}(z)|\leq M|\Re h_{nk}^{(j)}(z)|,\qquad z\in D_{r}\cup\Delta_{r},\; n\geq1,\;1\leq k\leq k_{n},\; j=1,2.
\]

\item[(\textbf{P4})] Given $r\in(0,1)$, the two approximations 
\[
h_{nk}^{(j)}(z)=o(1),\quad\eta_{\nu_{nk}^{(j)}}^{-1}(z)/z-1=h_{nk}^{(j)}(z)[1+o(1)]\qquad(n\rightarrow\infty)
\]
hold uniformly for $z\in D_{r}\cup\Delta_{r}$ and for $1\leq k\leq k_{n}$.
\end{enumerate}
We first derive some useful estimates. Note that the estimate (2)
below has already appeared in \cite{BerSerban,BerJC} and is a direct
consequence of (\textbf{P4}). We will verify (1) and (3) only. 
\begin{prop}
Given an infinitesimal array $\{\mu_{nk}\}_{n,k}\subset\mathscr{P}_{\mathbb{T}^{2}}^{\times}$
and $r\in(0,1)$, we have:
\begin{enumerate}
\item 
\[
\lim_{n\rightarrow\infty}\max_{1\leq k\leq k_{n}}\sup_{(z,w)\in\Omega_{r}}\left|\frac{(1-z)(1-w)}{zw}\psi_{\mu_{nk}}(z,w)-1\right|=0.
\]

\item For sufficiently large $n$, the inverse $\eta_{\mu_{nk}^{(j)}}^{-1}$
is defined in the set $D_{r}\cup\Delta_{r}$ and 
\[
\lim_{n\rightarrow\infty}\max_{1\leq k\leq k_{n}}\sup_{z\in D_{r}\cup\Delta_{r}}\left|\eta_{\mu_{nk}^{(j)}}^{-1}(z)-z\right|=0,\quad j=1,2.
\]

\item For sufficiently large $n$, the map $\Sigma_{\mu_{nk}}$ and its
principal logarithm $\log\Sigma_{\mu_{nk}}$ are defined in the product
set $\Omega_{r}$ and satisfies 
\begin{eqnarray*}
\log\Sigma_{\mu_{nk}}(z,w) & = & (1-zw)\left\{ h_{nk}^{(1)}(z)\varepsilon_{nk}^{(1)}(z,w)+h_{nk}^{(2)}(w)\varepsilon_{nk}^{(2)}(z,w)\right.\\
 &  & \left.+\left[\int_{\mathbb{T}^{2}}\frac{(1-s)(1-t)}{(1-zs)(1-wt)}\, d\nu_{nk}(s,t)\right][1+\varepsilon_{nk}^{(3)}(z,w)]\right\} 
\end{eqnarray*}
for $(z,w)\in\Omega_{r}$, where 
\[
\lim_{n\rightarrow\infty}\max_{1\leq k\leq k_{n}}\sup_{(z,w)\in\Omega_{r}}\left|\varepsilon_{nk}^{(j)}(z,w)\right|=0,\qquad j=1,2,3.
\]

\end{enumerate}
\end{prop}
\begin{proof}
We begin with the identity
\begin{eqnarray*}
\frac{(1-z)(1-w)}{zw}\psi_{\mu_{nk}}(z,w)-1 & = & -\int_{\mathbb{T}}\frac{1-s}{1-zs}\, d\mu_{nk}^{(1)}(s)-\int_{\mathbb{T}}\frac{1-t}{1-wt}\, d\mu_{nk}^{(2)}(t)\\
 &  & +\int_{\mathbb{T}^{2}}\frac{(1-s)(1-t)}{(1-zs)(1-wt)}\, d\mu_{nk}(s,t).
\end{eqnarray*}
Observe that for any $\varepsilon>0$, $(z,w)\in\Omega_{r}$, and
$j=1,2$, we have 
\begin{multline*}
\left|\int_{\mathbb{T}}\frac{1-x}{1-zx}\, d\mu_{nk}^{(j)}(x)\right|\leq\int_{|x-1|<\varepsilon}\frac{|1-x|}{|1-zx|}\, d\mu_{nk}^{(j)}(x)+\int_{|x-1|\geq\varepsilon}\frac{|1-x|}{|1-zx|}\, d\mu_{nk}^{(j)}(x)\\
\leq\frac{\varepsilon}{1-r}+\frac{2}{1-r}\max_{1\leq k\leq k_{n}}\mu_{nk}^{(j)}(\{x\in\mathbb{T}:|x-1|\geq\varepsilon\})
\end{multline*}
and
\begin{multline*}
\left|\int_{\mathbb{T}^{2}}\frac{(1-s)(1-t)}{(1-zs)(1-wt)}\, d\mu_{nk}(s,t)\right|\leq\frac{2}{(1-r)^{2}}\int_{\mathbb{T}}|1-s|\, d\mu_{nk}^{(1)}(s)\\
\leq\frac{2\varepsilon}{(1-r)^{2}}+\frac{2}{(1-r)^{2}}\max_{1\leq k\leq k_{n}}\mu_{nk}^{(1)}(\{s\in\mathbb{T}:|s-1|\geq\varepsilon\}),
\end{multline*}
whence the uniform estimate (1) holds. 

We now focus on the proof of (3). First, note that $\Sigma_{\mu_{nk}}=\Sigma_{\nu_{nk}}$
by rotational invariance. We shall apply the uniform estimates (1)
and (2) to the infinitesimal array $\{\nu_{nk}\}_{n,k}$. 

For notational convenience, we write $\psi_{nk}=\psi_{\nu_{nk}}$,
$\Sigma_{nk}=\Sigma_{\nu_{nk}}=\Sigma_{\mu_{nk}}$, and $\eta_{jnk}^{-1}=\eta_{\nu_{nk}^{(j)}}^{-1}$.
Also, we use the symbol $a_{n}\approx b_{n}$ to indicate $\lim_{n\rightarrow\infty}a_{n}/b_{n}=1$.
We conclude from (1) and (2) that 
\[
\psi_{nk}\left(\eta_{1nk}^{-1}(z),\eta_{2nk}^{-1}(w)\right)\approx\frac{zw}{(1-z)(1-w)}\qquad(n\rightarrow\infty)
\]
uniformly for $1\leq k\leq k_{n}$ and for any $(z,w)\in\Omega_{r}$.
By (\ref{eq:2.5}), we also have the uniform approximation 
\begin{equation}
H_{\nu_{nk}}\left(\eta_{1nk}^{-1}(z),\eta_{2nk}^{-1}(w)\right)\approx\frac{1}{(1-z)(1-w)}\qquad(n\rightarrow\infty)\label{eq:3.2}
\end{equation}
of the same nature. Therefore the transform $\Sigma_{nk}$ is well-defined
in $\Omega_{r}$ for large $n$. 

The preceding approximations also show that the limit $\lim_{n\rightarrow\infty}\Sigma_{nk}=1$
holds uniformly in $k$ and in $\Omega_{r}$. Hence the principal
logarithm $\log\Sigma_{nk}$ of $\Sigma_{nk}$ exists in $\Omega_{r}$
when $n$ is sufficiently large. We will only consider these large
$n$'s from now on. On the other hand, the fact that $\log x\approx(x-1)$
as $x\rightarrow1$ implies 
\[
\log\Sigma_{nk}\approx\left[\Sigma_{nk}-1\right]\qquad(n\rightarrow\infty)
\]
uniformly in $k$ and in the set $\Omega_{r}$. 

We should derive an estimate for $\Sigma_{nk}-1$. As $n\rightarrow\infty$,
observe that 
\begin{eqnarray*}
\Sigma_{nk}(z,w)-1 & = & \frac{\psi_{nk}\left(\eta_{1nk}^{-1}(z),\eta_{2nk}^{-1}(w)\right)-zwH_{\nu_{nk}}\left(\eta_{1nk}^{-1}(z),\eta_{2nk}^{-1}(w)\right)}{zwH_{\nu_{nk}}\left(\eta_{1nk}^{-1}(z),\eta_{2nk}^{-1}(w)\right)}\\
 & \approx & \frac{(1-z)(1-w)}{zw}\left[(1-zw)\psi_{nk}\left(\eta_{1nk}^{-1}(z),\eta_{2nk}^{-1}(w)\right)-\frac{zw(1-zw)}{(1-z)(1-w)}\right]\\
 & = & (1-zw)\left[\frac{(1-z)(1-w)}{zw}\psi_{nk}\left(\eta_{1nk}^{-1}(z),\eta_{2nk}^{-1}(w)\right)-1\right]
\end{eqnarray*}
in $\Omega_{r}$. Note that we have made use of (\ref{eq:2.5}) and
the estimate (\ref{eq:3.2}) in this calculation. 

Introduce 
\[
u_{nk}(z,w)=\frac{(1-z)(1-w)}{zw}\psi_{nk}\left(\eta_{1nk}^{-1}(z),\eta_{2nk}^{-1}(w)\right)-1,
\]
so that $\log\Sigma_{nk}(z,w)\approx(1-zw)u_{nk}(z,w)$ in $\Omega_{r}$.
Next, we use the identities 
\[
\frac{\eta_{1nk}^{-1}(z)s}{1-\eta_{1nk}^{-1}(z)s}=\frac{zs}{1-zs}\left[1+\frac{\eta_{1nk}^{-1}(z)/z-1}{1-\eta_{1nk}^{-1}(z)s}\right]
\]
and 
\[
\frac{\eta_{2nk}^{-1}(w)t}{1-\eta_{2nk}^{-1}(w)t}=\frac{wt}{1-wt}\left[1+\frac{\eta_{2nk}^{-1}(w)/w-1}{1-\eta_{2nk}^{-1}(w)t}\right]
\]
to get 
\begin{eqnarray*}
u_{nk}(z,w) & = & \int_{\mathbb{T}^{2}}\frac{(\eta_{1nk}^{-1}(z)/z-1)(1-z)(1-w)st}{(1-zs)(1-wt)(1-\eta_{1nk}^{-1}(z)s)}\, d\nu_{nk}(s,t)\\
 &  & +\int_{\mathbb{T}^{2}}\frac{(\eta_{2nk}^{-1}(w)/w-1)(1-z)(1-w)st}{(1-zs)(1-wt)(1-\eta_{2nk}^{-1}(w)t)}\, d\nu_{nk}(s,t)\\
 &  & +\int_{\mathbb{T}^{2}}\frac{(\eta_{1nk}^{-1}(z)/z-1)(\eta_{2nk}^{-1}(w)/w-1)(1-z)(1-w)st}{(1-zs)(1-wt)(1-\eta_{1nk}^{-1}(z)s)(1-\eta_{2nk}^{-1}(w)t)}\, d\nu_{nk}(s,t)\\
 &  & +\int_{\mathbb{T}^{2}}\frac{st-1+zs(1-t)+wt(1-s)}{(1-zs)(1-wt)}\, d\nu_{nk}(s,t).
\end{eqnarray*}

The infinitesimality of $\{\nu_{nk}\}_{n,k}$ and the estimate (2)
yields 
\[
\int_{\mathbb{T}^{2}}\frac{(1-z)(1-w)st}{(1-zs)(1-wt)(1-\eta_{1nk}^{-1}(z)s)}\, d\nu_{nk}(s,t)\approx\frac{1}{1-z}\qquad(n\rightarrow\infty).
\]
Combining with the property (\textbf{P4}), we obtain that 
\[
\int_{\mathbb{T}^{2}}\frac{(\eta_{1nk}^{-1}(z)/z-1)(1-z)(1-w)st}{(1-zs)(1-wt)(1-\eta_{1nk}^{-1}(z)s)}\, d\nu_{nk}(s,t)\approx\frac{h_{nk}^{(1)}(z)}{1-z}
\]
as $n\rightarrow\infty$. Similarly, one also has 
\[
\int_{\mathbb{T}^{2}}\frac{(\eta_{2nk}^{-1}(w)/w-1)(1-z)(1-w)st}{(1-zs)(1-wt)(1-\eta_{2nk}^{-1}(w)t)}\, d\nu_{nk}(s,t)\approx\frac{h_{nk}^{(2)}(w)}{1-w}\qquad(n\rightarrow\infty).
\]
For the third integral in the decomposition of $u_{nk}$, we observe
that 
\begin{multline*}
\int_{\mathbb{T}^{2}}\frac{(\eta_{1nk}^{-1}(z)/z-1)(\eta_{2nk}^{-1}(w)/w-1)(1-z)(1-w)st}{(1-zs)(1-wt)(1-\eta_{1nk}^{-1}(z)s)(1-\eta_{2nk}^{-1}(w)t)}\, d\nu_{nk}(s,t)\\
\approx\left[\frac{\eta_{1nk}^{-1}(z)/z-1}{1-z}\right]\frac{\eta_{2nk}^{-1}(w)/w-1}{1-w}=\left[\frac{h_{nk}^{(1)}(z)}{1-z}\right]\cdot o(1)
\end{multline*}
as $n\rightarrow\infty$. Finally, note that
\begin{multline*}
\int_{\mathbb{T}^{2}}\frac{st-1+zs(1-t)+wt(1-s)}{(1-zs)(1-wt)}\, d\nu_{nk}(s,t)\\
=\int_{\mathbb{T}^{2}}\frac{(1-s)(1-t)}{(1-zs)(1-wt)}\, d\nu_{nk}(s,t)-\frac{h_{nk}^{(1)}(z)}{1-z}-\frac{h_{nk}^{(2)}(w)}{1-w}.
\end{multline*}
Since $1/|1-z|<1/(1-r)$ for $z\in D_{r}\cup\Delta_{r}$, we conclude
from these findings that 
\begin{eqnarray*}
u_{nk}(z,w) & \approx & h_{nk}^{(1)}(z)\cdot o(1)+h_{nk}^{(2)}(w)\cdot o(1)\\
 &  & +\int_{\mathbb{T}^{2}}\frac{(1-s)(1-t)}{(1-zs)(1-wt)}\, d\nu_{nk}(s,t)
\end{eqnarray*}
as $n\rightarrow\infty$ uniformly in $k$ and for any $(z,w)\in\Omega_{r}$.
The estimate (3) now follows. 
\end{proof}
We need an elementary fact.
\begin{lem}
Let $\{z_{nk}\}_{n,k}$ and $\{\varepsilon_{nk}\}_{n,k}$ be two triangular
arrays of complex numbers. Suppose that there exists a universal constant
$M>0$, independent of $n$ and $k$, such that 
\[
|\Im z_{nk}|\leq M|\Re z_{nk}|,\qquad n\geq1,\;1\leq k\leq k_{n},
\]
and that $\sup_{n\geq1}\sum_{k=1}^{k_{n}}|\Re z_{nk}|<\infty$ and
$\lim_{n\rightarrow\infty}\max_{1\leq k\leq k_{n}}|\varepsilon_{nk}|=0$.
Then we have $\sum_{k=1}^{k_{n}}z_{nk}\varepsilon_{nk}\rightarrow0$
as $n\rightarrow\infty$. \end{lem}
\begin{proof}
This result follows from the following observation:
\[
\left|\sum_{k=1}^{k_{n}}z_{nk}\varepsilon_{nk}\right|\leq(1+M)\cdot\max_{1\leq k\leq k_{n}}|\varepsilon_{nk}|\cdot\sup_{n\geq1}\sum_{k=1}^{k_{n}}|\Re z_{nk}|\rightarrow0\qquad(n\rightarrow\infty).
\]
\end{proof}
Recall that the marginal of the bi-free convolution $\delta_{\lambda_{n}}\boxtimes\boxtimes\mu_{n1}\boxtimes\boxtimes\mu_{n2}\boxtimes\boxtimes\cdots\boxtimes\boxtimes\mu_{nk_{n}}$
is the usual free convolution $\mu_{n}=\delta_{\pi_{j}(\lambda_{n})}\boxtimes\mu_{n1}^{(j)}\boxtimes\mu_{n2}^{(j)}\boxtimes\cdots\boxtimes\mu_{nk_{n}}^{(j)}$.
We now review the limit theorems for $\{\mu_{n}\}_{n=1}^{\infty}$.
First, any weak limit point $\nu_{j}$ of $\{\mu_{n}\}_{n=1}^{\infty}$
must be \emph{$\boxtimes$-infinitely divisible} on $\mathbb{T}$
\cite{BerSerban}. If $\nu_{j}\in\mathscr{P}_{\mathbb{T}}^{\times}$,
then we have 
\[
\eta_{\nu_{j}}^{-1}(z)=\gamma_{j}\, z\,\exp\left(\int_{\mathbb{T}}\frac{1+xz}{1-xz}\, d\sigma_{j}(x)\right),\qquad z\in\mathbb{D},
\]
where $\gamma_{j}\in\mathbb{T}$ and $\sigma_{j}\in\mathscr{M}_{\mathbb{T}}$.
The parameters $\gamma_{j}$ and $\sigma_{j}$ (called the \emph{L\'{e}vy
parameters}) are uniquely associated with the limiting measure $\nu_{j}$.
Indeed, we have 
\[
\gamma_{j}=\frac{|\text{m}(\nu_{j})|}{\text{m}(\nu_{j})},\qquad d\sigma_{j}(e^{-i\theta})=\text{w*-}\lim_{r\uparrow1}\frac{1}{2\pi}\log\frac{|\eta_{\nu_{j}}^{-1}(re^{i\theta})|}{r}\, d\theta.
\]
Conversely, given a pair $(\gamma_{j},\sigma_{j})$ of L\'{e}vy parameters,
the above exponential integral formula determines a unique $\boxtimes$-infinitely
divisible law $\nu_{j}$ in $\mathscr{P}_{\mathbb{T}}^{\times}$.
We shall write $\nu_{j}=\nu_{\boxtimes}^{\gamma_{j},\sigma_{j}}$
to indicate this correspondence. 

Secondly, the sequence $\mu_{n}$ converges weakly to an $\boxtimes$-infinitely
divisible law $\nu_{\boxtimes}^{\gamma_{j},\sigma_{j}}$ on $\mathbb{T}$
if and only if the system 
\[
\begin{cases}
\sum\nolimits _{k=1}^{k_{n}}(1-\Re x)\, d\nu_{nk}^{(j)}\Rightarrow\sigma_{j},\\
\lim_{n\rightarrow\infty}\overline{\pi_{j}(\lambda_{n})}\,\exp\left(-i\sum\nolimits _{k=1}^{k_{n}}\left[\int_{\mathbb{T}}\Im x\, d\nu_{nk}^{(j)}(x)+\arg b_{nk}^{(j)}\right]\right)=\gamma_{j}
\end{cases}
\]
of weak and numerical limits holds. (Here the measure $\nu_{nk}^{(j)}$
refers to the rotation (\ref{eq:3.1}).) 
\begin{lem}
Let $\{\mu_{nk}\}_{n,k}$ be an infinitesimal array in $\mathscr{P}_{\mathbb{T}^{2}}^{\times}$,
and let $\{\nu_{nk}\}_{n,k}$ be the accompanying array defined by
the marginal rotation \emph{(\ref{eq:3.1})}. Suppose that there exists
a sequence $\lambda_{n}\in\mathbb{T}^{2}$ such that 
\[
\delta_{\pi_{j}(\lambda_{n})}\boxtimes\mu_{n1}^{(j)}\boxtimes\mu_{n2}^{(j)}\boxtimes\cdots\boxtimes\mu_{nk_{n}}^{(j)}\Rightarrow\nu_{\boxtimes}^{\gamma_{j},\sigma_{j}},\qquad j=1,2.
\]
Then we have: 
\begin{enumerate}
\item The two sequences $S=\{\sum\nolimits _{k=1}^{k_{n}}(1-\Re s)\, d\nu_{nk}:n\geq1\}$
and $T=\{\sum\nolimits _{k=1}^{k_{n}}(1-\Re t)\, d\nu_{nk}:n\geq1\}$
of positive Borel measures on $\mathbb{T}^{2}$ have uniformly bounded
total variation norms, and the sequence $C=\{\sum\nolimits _{k=1}^{k_{n}}\int_{\mathbb{T}^{2}}\Im s\Im t\, d\nu_{nk}(s,t):n\geq1\}$
of real numbers is bounded. 
\item Let $\Omega_{r}$ be a common domain of definition for all $\Sigma_{\mu_{nk}}$
and $U\subseteq\Omega_{r}$ be an arbitrary open subset. The pointwise
limit 
\[
F(z,w)=\lim_{n\rightarrow\infty}\Pi_{k=1}^{k_{n}}\Sigma_{\mu_{nk}}(z,w)
\]
exists for $(z,w)\in U$ if and only if and the pointwise limit 
\begin{equation}
G(z,w)=\lim_{n\rightarrow\infty}\exp\left(\sum_{k=1}^{k_{n}}\int_{\mathbb{T}^{2}}\frac{(1-zw)(1-s)(1-t)}{(1-zs)(1-wt)}\, d\nu_{nk}(s,t)\right)\label{eq:3.3}
\end{equation}
exists for $(z,w)\in U$. In this case, we also have $F=G$ in $U$. 
\end{enumerate}
\end{lem}
\begin{proof}
The fact that $S$ and $T$ have uniformly bounded total variation
norms is rather obvious; for example, the 1-D free limit theorem yields
\[
\int_{\mathbb{T}^{2}}\sum\nolimits _{k=1}^{k_{n}}(1-\Re s)\, d\nu_{nk}(s,t)=\int_{\mathbb{T}}\sum\nolimits _{k=1}^{k_{n}}(1-\Re s)\, d\nu_{nk}^{(1)}(s)\leq2\sigma_{1}(\mathbb{T})
\]
for sufficiently large $n$. Meanwhile, the estimate 
\begin{eqnarray*}
\int_{\mathbb{T}^{2}}|\Im s\Im t|\, d\nu_{nk}(s,t) & = & \int_{\mathbb{T}^{2}}\sqrt{(1+\Re s)(1-\Re s)}\sqrt{(1+\Re t)(1-\Re t)}\, d\nu_{nk}(s,t)\\
 & \leq & \int_{\mathbb{T}^{2}}2\sqrt{1-\Re s}\sqrt{1-\Re t}\, d\nu_{nk}(s,t)\\
 & \leq & \int_{\mathbb{T}^{2}}(1-\Re s)+(1-\Re t)\, d\nu_{nk}(s,t)\\
 & = & \int_{\mathbb{T}}1-\Re s\, d\nu_{nk}^{(1)}(s)+\int_{\mathbb{T}}1-\Re t\, d\nu_{nk}^{(2)}(t)
\end{eqnarray*}
shows that $C$ is a bounded sequence. So, (1) is proved.

We know from Proposition 3.1 (3) that the pointwise convergence 
\[
\Pi_{k=1}^{k_{n}}\Sigma_{\mu_{nk}}(z,w)=\exp\left(\sum\nolimits _{k=1}^{k_{n}}\log\Sigma_{\mu_{nk}}(z,w)\right)\rightarrow F(z,w)
\]
holds in the set $U$ if and only if the pointwise limit 
\begin{equation}
G(z,w)=\lim_{n\rightarrow\infty}\exp\left(\sum\nolimits _{k=1}^{k_{n}}u_{nk}(z,w)\right)\label{eq:3.4}
\end{equation}
exists in $U$, where the function 
\begin{multline*}
u_{nk}(z,w)=(1-zw)\left\{ \sum_{k=1}^{k_{n}}h_{nk}^{(1)}(z)\varepsilon_{nk}^{(1)}(z,w)+\sum_{k=1}^{k_{n}}h_{nk}^{(2)}(w)\varepsilon_{nk}^{(2)}(z,w)\right.\\
\left.+\sum_{k=1}^{k_{n}}\int_{\mathbb{T}^{2}}\frac{(1-s)(1-t)(1+\varepsilon_{nk}^{(3)}(z,w))}{(1-zs)(1-wt)}\, d\nu_{nk}(s,t)\right\} .
\end{multline*}
Of course, we will have $F=G$ in $U$ if any of the two limits holds.

We now seek a re-casting of (\ref{eq:3.4}) into the desired limit
(\ref{eq:3.3}). We first notice that 
\[
\sum_{k=1}^{k_{n}}\Re h_{nk}^{(1)}(z)=\sum_{k=1}^{k_{n}}\int_{\mathbb{T}}\Re\left[\frac{1+sz}{1-sz}\right](1-\Re s)\, d\nu_{nk}^{(1)}(s)\leq\frac{2(1+r)}{1-r}\sigma_{1}(\mathbb{T})
\]
for any $z\in D_{r}$ when $n$ is sufficiently large. Then (\textbf{P2})
implies that 
\[
\sup_{n\geq1}\sum_{k=1}^{k_{n}}\left|\Re h_{nk}^{(1)}(z)\right|<\infty,\qquad z\in D_{r}\cup\Delta_{r}.
\]
Similarly, we also have $\sup_{n\geq1}\sum_{k=1}^{k_{n}}\left|\Re h_{nk}^{(2)}(w)\right|<\infty$
for $w\in D_{r}\cup\Delta_{r}$. By (\textbf{P3}) and Lemma 3.2, we
get 
\[
\lim_{n\rightarrow\infty}\sum_{k=1}^{k_{n}}h_{nk}^{(1)}(z)\varepsilon_{nk}^{(1)}(z,w)=0=\lim_{n\rightarrow\infty}\sum_{k=1}^{k_{n}}h_{nk}^{(2)}(w)\varepsilon_{nk}^{(2)}(z,w)
\]
for $(z,w)\in U$, and so these terms do not make any contribution
to the limit (\ref{eq:3.4}). As a result, (\ref{eq:3.4}) is equivalent
to 
\[
\lim_{n\rightarrow\infty}\exp\left(\sum_{k=1}^{k_{n}}\int_{\mathbb{T}^{2}}\frac{(1-zw)(1-s)(1-t)(1+\varepsilon_{nk}^{(3)}(z,w))}{(1-zs)(1-wt)}\, d\nu_{nk}(s,t)\right)=G(z,w)
\]
in $U$. 

Next, the previous estimate for the sequence $C$ leads to 
\begin{multline*}
\left|\int_{\mathbb{T}^{2}}\frac{(1-s)(1-t)}{(1-zs)(1-wt)}\, d\nu_{nk}(s,t)\right|\\
\leq\frac{4}{(1-r)^{2}}\left\{ \int_{\mathbb{T}}1-\Re s\, d\nu_{nk}^{(1)}(s)+\int_{\mathbb{T}}1-\Re t\, d\nu_{nk}^{(2)}(t)\right\} ,
\end{multline*}
which tells us that 
\begin{multline*}
\left|\sum_{k=1}^{k_{n}}\int_{\mathbb{T}^{2}}\frac{(1-s)(1-t)\varepsilon_{nk}^{(3)}(z,w)}{(1-zs)(1-wt)}\, d\nu_{nk}(s,t)\right|\\
\leq\frac{8\left\{ \sigma_{1}(\mathbb{T})+\sigma_{2}(\mathbb{T})\right\} }{(1-r)^{2}}\max_{1\leq k\leq k_{n}}\sup_{(z,w)\in\Omega_{r}}\left|\varepsilon_{nk}^{(3)}(z,w)\right|\rightarrow 0
\end{multline*}
in $U$ as $n\rightarrow\infty$. Thus, we see that the limit (\ref{eq:3.4})
is eventually re-casted to (\ref{eq:3.3}). 
\end{proof}
We now prove our main result. For the rest of the paper, we introduce
the function 
\[
f(z,w)=\frac{1-zw}{(1-z)(1-w)},\qquad(z,w)\in(\mathbb{C}\setminus\mathbb{T})^{2}.
\]

\begin{thm}
Let $\{\mu_{nk}\}_{n,k}$ be an infinitesimal array in $\mathscr{P}_{\mathbb{T}^{2}}^{\times}$
and let $\{\lambda_{n}\}_{n=1}^{\infty}$ be a sequence of points
on $\mathbb{T}^{2}$. Denote by $\{\nu_{nk}\}_{n,k}$ the accompanying
array of $\{\mu_{nk}\}_{n,k}$ according to \emph{(\ref{eq:3.1})}.
Then the bi-free convolutions 
\[
\mu_{n}=\delta_{\lambda_{n}}\boxtimes\boxtimes\mu_{n1}\boxtimes\boxtimes\mu_{n2}\boxtimes\boxtimes\cdots\boxtimes\boxtimes\mu_{nk_{n}},\qquad n\geq1,
\]
converge weakly to a probability measure in $\mathscr{P}_{\mathbb{T}^{2}}^{\times}$
if and only if there exist two Borel measures $\rho_{1},\rho_{2}\in\mathscr{M}_{\mathbb{T}^{2}}$,
two complex numbers $\gamma_{1},\gamma_{2}\in\mathbb{T}$, and a constant
$a\in\mathbb{R}$ such that the following weak and numerical convergences
\begin{equation}
\begin{cases}
\sum_{k=1}^{k_{n}}(1-\Re s)\, d\nu_{nk}\Rightarrow\rho_{1},\\
\sum_{k=1}^{k_{n}}(1-\Re t)\, d\nu_{nk}\Rightarrow\rho_{2},\\
\lim_{n\rightarrow\infty}\sum_{k=1}^{k_{n}}\int_{\mathbb{T}^{2}}\Im s\Im t\, d\nu_{nk}(s,t)=a,\\
\lim_{n\rightarrow\infty}\overline{\pi_{j}(\lambda_{n})}\,\exp\left(-i\sum\nolimits _{k=1}^{k_{n}}\left[\int_{\mathbb{T}}\Im x\, d\nu_{nk}^{(j)}(x)+\arg b_{nk}^{(j)}\right]\right)=\gamma_{j}, & j=1,2,
\end{cases}\label{eq:3.5}
\end{equation}
hold simultaneously. In this case, if $\mu_{n}\Rightarrow\nu$ then
the limit $\nu$ is determined by the conditions: 

\[
\begin{cases}
\Sigma_{\nu}(z,w)=\exp\left(f(z,w)F_{\rho_{1},\rho_{2},a}(z,w)\right), & (z,w)\in(\mathbb{C}\setminus\mathbb{T})^{2};\\
\nu^{(j)}=\nu_{\boxtimes}^{\gamma_{j},\rho_{j}\circ\pi_{j}^{-1}}, & j=1,2.
\end{cases}
\]
Here the function $F_{\rho_{1},\rho_{2},a}$ is defined by 
\begin{eqnarray*}
F_{\rho_{1},\rho_{2},a}(z,w) & = & \int_{\mathbb{T}^{2}}\frac{1+zs}{1-zs}\frac{1+wt}{1-wt}(1-\Re t)\, d\rho_{1}(s,t)\\
 &  & -i\int_{\mathbb{T}^{2}}\frac{(1+zs)\Im t}{1-zs}\, d\rho_{1}(s,t)-i\int_{\mathbb{T}^{2}}\frac{(1+wt)\Im s}{1-wt}\: d\rho_{2}(s,t)-a.
\end{eqnarray*}
\end{thm}
\begin{proof}
Assume $\mu_{n}\Rightarrow\nu$ for some $\nu\in\mathscr{P}_{\mathbb{T}^{2}}^{\times}$,
and say, $\mu_{n}^{(j)}\Rightarrow\nu^{(j)}=\nu_{\boxtimes}^{\gamma_{j},\sigma_{j}}$
for $j=1,2$. Lemma 3.3 (1) shows that the two sequences $S$ and
$T$ of measures have weak-star limit points in the set $\mathscr{M}_{\mathbb{T}^{2}}$,
and the numerical sequence $C$ has limit points in $\mathbb{R}$.
Let $(\rho_{1},\rho_{2},a)$ and $(\rho_{1}^{\prime},\rho_{2}^{\prime},a^{\prime})$
be any two triples of such subsequential limits. Thus, there exist
two infinite subsets $A$ and $A^{\prime}$ in $\mathbb{N}$ such
that 
\[
\begin{cases}
\sum_{k=1}^{k_{n}}(1-\Re s)\, d\nu_{nk}\Rightarrow\rho_{1},\\
\sum_{k=1}^{k_{n}}(1-\Re t)\, d\nu_{nk}\Rightarrow\rho_{2},
\end{cases}\text{and}\quad\sum_{k=1}^{k_{n}}\int_{\mathbb{T}^{2}}\Im s\Im t\, d\nu_{nk}(s,t)\rightarrow a\quad(n\rightarrow\infty,\: n\in A),
\]
and 
\[
\begin{cases}
\sum_{k=1}^{k_{n}}(1-\Re s)\, d\nu_{nk}\Rightarrow\rho_{1}^{\prime},\\
\sum_{k=1}^{k_{n}}(1-\Re t)\, d\nu_{nk}\Rightarrow\rho_{2}^{\prime},
\end{cases}\text{and}\quad\sum_{k=1}^{k_{n}}\int_{\mathbb{T}^{2}}\Im s\Im t\, d\nu_{nk}(s,t)\rightarrow a^{\prime}\quad(n\rightarrow\infty,\: n\in A^{\prime}).
\]
Our task here is to prove $\rho_{1}=\rho_{1}^{\prime}$, $\rho_{2}=\rho_{2}^{\prime}$,
$a=a^{\prime}$, as well as the integral representation of $\Sigma_{\nu}$.
Note that the convergence to $\gamma_{j}$ in (\ref{eq:3.5}) already
holds by the 1-D free limit theorems.

By Proposition 3.1 (3), we can choose a universal domain of definition
$\Omega_{r}$ for all $\Sigma_{\mu_{nk}}$ and $\Sigma_{\nu}$. Corollary
2.9 and Theorem 2.10 imply the pointwise convergence 
\[
\Pi_{k=1}^{k_{n}}\Sigma_{\mu_{nk}}=\Sigma_{\mu_{n}}\rightarrow\Sigma_{\nu}
\]
in the set $\Omega_{r}$. By virtue of Lemma 3.3 (2), this means 
\[
\lim_{n\rightarrow\infty}\exp\left(\sum_{k=1}^{k_{n}}\int_{\mathbb{T}^{2}}\frac{(1-zw)(1-s)(1-t)}{(1-zs)(1-wt)}\, d\nu_{nk}(s,t)\right)=\Sigma_{\nu}(z,w)
\]
for $(z,w)\in\Omega_{r}$. Using the formula 
\[
\frac{(1-z)(1-s)}{1-zs}=\frac{1+zs}{1-zs}(1-\Re s)-i\Im s,
\]
we can re-write the integrand in the last limit as
\begin{eqnarray*}
\frac{(1-zw)(1-s)(1-t)}{(1-zs)(1-wt)} & = & f(z,w)\left\{ \frac{1+zs}{1-zs}\frac{1+wt}{1-wt}(1-\Re t)(1-\Re s)\right.\\
 &  & \left.-i\frac{1+zs}{1-zs}(\Im t)(1-\Re s)-i\frac{1+wt}{1-wt}(\Im s)(1-\Re t)-\Im s\Im t\right\} .
\end{eqnarray*}
After passing to the subsequential limits $(\rho_{1},\rho_{2},a)$
and $(\rho_{1}^{\prime},\rho_{2}^{\prime},a^{\prime})$, we obtain
the identity $\exp\left(f\cdot F_{\rho_{1},\rho_{2},a}\right)=\Sigma_{\nu}=\exp\left(f\cdot F_{\rho_{1}^{\prime},\rho_{2}^{\prime},a^{\prime}}\right)$
in the set $\Omega_{r}$. Since both $\exp\left(f\cdot F_{\rho_{1},\rho_{2},a}\right)$
and $\exp\left(f\cdot F_{\rho_{1}^{\prime},\rho_{2}^{\prime},a^{\prime}}\right)$
are already holomorphic in $(\mathbb{C}\setminus\mathbb{T})^{2}$
and $\Omega_{r}$ is an open set intersecting all four connected components
of $(\mathbb{C}\setminus\mathbb{T})^{2}$, the uniqueness principle
yields $\exp\left(f\cdot F_{\rho_{1},\rho_{2},a}\right)=\exp\left(f\cdot F_{\rho_{1}^{\prime},\rho_{2}^{\prime},a^{\prime}}\right)$
in all of $(\mathbb{C}\setminus\mathbb{T})^{2}$. When focusing on
the connected components $\mathbb{D}^{2}$ and $\mathbb{D}\times\Delta_{1}$,
we see that there exists a unique integer $k$ such that 
\begin{equation}
f(z,w)F_{\rho_{1},\rho_{2},a}(z,w)=f(z,w)F_{\rho_{1}^{\prime},\rho_{2}^{\prime},a^{\prime}}(z,w)+2k\pi i,\qquad(z,w)\in\mathbb{D}^{2},\label{eq:3.6}
\end{equation}
and 
\begin{equation}
f(z,w)F_{\rho_{1},\rho_{2},a}(z,w)=f(z,w)F_{\rho_{1}^{\prime},\rho_{2}^{\prime},a^{\prime}}(z,w),\qquad(z,w)\in\mathbb{D}\times\Delta_{1}.\label{eq:3.7}
\end{equation}
Note that the phase transition constant in (\ref{eq:3.7}) is zero,
because $\lim_{z\rightarrow0,\,|w|\rightarrow\infty}f(z,w)=0$ and
both limits 
\[
\lim_{z\rightarrow0,\,|w|\rightarrow\infty}F_{\rho_{1},\rho_{2},a}(z,w),\quad\lim_{z\rightarrow0,\,|w|\rightarrow\infty}F_{\rho_{1}^{\prime},\rho_{2}^{\prime},a^{\prime}}(z,w)
\]
exist in $\mathbb{C}$. 

Next, we argue that $k=0$. Indeed, for $(z,w)\in\mathbb{D}^{2}$ and
$w\neq0$, we first evaluate (\ref{eq:3.6}) and (\ref{eq:3.7}) at
$(z,w)$ and $(z,1/\overline{w})$ respectively and then consider
the difference ``$\text{(3.6)}-\text{(3.7)}$''. We get 
\begin{multline*}
\int_{\mathbb{T}^{2}}\frac{1+zs}{1-zs}P_{w}(\overline{t})(1-\Re t)\, d\rho_{1}(s,t)-i\int_{\mathbb{T}^{2}}P_{w}(\overline{t})\Im s\: d\rho_{2}(s,t)\\
=\int_{\mathbb{T}^{2}}\frac{1+zs}{1-zs}P_{w}(\overline{t})(1-\Re t)\, d\rho_{1}^{\prime}(s,t)-i\int_{\mathbb{T}^{2}}P_{w}(\overline{t})\Im s\: d\rho_{2}^{\prime}(s,t)+\frac{k\pi i}{f(z,w)},
\end{multline*}
where 
\[
P_{w}(x)=\Re\left[\frac{x+w}{x-w}\right],\qquad w\in\mathbb{D},\; x\in\mathbb{T},
\]
denotes the usual Poisson kernel in the disk $\mathbb{D}$. (Note
that the function $f$ does not have any zero in $\mathbb{D}^{2}$.)
Taking further the real part of the previous identity, we obtain
\begin{equation}
\int_{\mathbb{T}^{2}}P_{z}(\overline{s})P_{w}(\overline{t})(1-\Re t)\, d\rho_{1}(s,t)=\int_{\mathbb{T}^{2}}P_{z}(\overline{s})P_{w}(\overline{t})(1-\Re t)\, d\rho_{1}^{\prime}(s,t)-\Im\frac{k\pi}{f(z,w)}\label{eq:3.8}
\end{equation}
for any $(z,w)\in\mathbb{D}^{2}$, $w\neq0$ (actually, for $w=0$
as well, because both sides of (\ref{eq:3.8}) are continuous in $\mathbb{D}^{2}$).

Now, plug in $z=0$ at (\ref{eq:3.8}), we arrive at the formula 
\begin{equation}
\int_{\mathbb{T}}P_{w}(\overline{t})\, d\tau_{1}(t)=\int_{\mathbb{T}}P_{w}(\overline{t})\, d\tau_{2}(t)+k\pi\Im w,\label{eq:3.9}
\end{equation}
where the one-dimensional measures $\tau_{1}$, $\tau_{2}$ are given
by 
\[
d\tau_{1}(t)=(1-\Re t)\, d[\rho_{1}\circ\pi_{2}^{-1}](t),\quad d\tau_{2}(t)=(1-\Re t)\, d[\rho_{1}^{\prime}\circ\pi_{2}^{-1}](t).
\]

Let $\varepsilon\in(0,\pi/4)$ be arbitrary but fixed. For any $n>1/\varepsilon$,
we choose a function $f_{n}\in C(\mathbb{T})$ satisfying $0\leq f_{n}\leq1$,
$f_{n}(e^{i\theta})=1$ for $\theta\in[\varepsilon,\pi/4]$, and $f_{n}(e^{i\theta})=0$
for $\theta\in(-\pi,\varepsilon-n^{-1}]\cup[\pi/4+n^{-1},\pi]$. Let \[p_{j}(re^{i\theta})=\int_{\mathbb{T}}P_{re^{i\theta}}(\overline{t})\, d\tau_{j}(t)\] be the Poisson integral of the measure $d\tau_{j}(1/t)$, so that the Lebesgue absolutely continuous measures  $d\mu_{jr}(\theta)=p_{j}(re^{i\theta})\,d\theta/2\pi$ tend weakly to $d\tau_{j}(1/t)$ as $r\rightarrow 1^{-}$. We now have 
\begin{eqnarray*}
\lim_{n\rightarrow\infty}\lim_{r\uparrow1}\int_{\mathbb{T}}\frac{f_{n}(e^{i\theta})}{1-\cos\theta}\,d\mu_{jr}(\theta) & = & \lim_{n\rightarrow\infty}\int_{-\pi}^{\pi}\frac{f_{n}(e^{i\theta})}{1-\cos\theta}\, d\tau_{j}(e^{-i\theta})\\
 & = & \begin{cases}
\rho_{1}\circ\pi_{2}^{-1}(A_{\varepsilon}), & j=1;\\
\rho_{1}^{\prime}\circ\pi_{2}^{-1}(A_{\varepsilon}), & j=2,
\end{cases}
\end{eqnarray*}
where the arc $A_{\varepsilon}=\{e^{i\theta}:\theta\in[-\pi/4,-\varepsilon]\}$.
Therefore, (\ref{eq:3.9}) yields
\[
\left|\frac{k}{2}\int_{\varepsilon}^{\pi/4}\frac{\sin\theta}{1-\cos\theta}\, d\theta\right|=\left|\rho_{1}\circ\pi_{2}^{-1}(A_{\varepsilon})-\rho_{1}^{\prime}\circ\pi_{2}^{-1}(A_{\varepsilon})\right|\leq\rho_{1}(\mathbb{T}^{2})+\rho_{1}^{\prime}(\mathbb{T}^{2})<\infty
\]
for all $\varepsilon\in(0,\pi/4)$, which is possible only if $k=0$. 

Since $k=0$, the equation (\ref{eq:3.8}) implies that $(1-\Re t)\, d\rho_{1}$
and $(1-\Re t)\, d\rho_{1}^{\prime}$ have the same Poisson integral
over $\mathbb{D}^{2}$ and so these two measures are the same on $\mathbb{T}^{2}$.
Meanwhile, the marginal weak convergence $\mu_{n}^{(1)}\Rightarrow\nu_{\boxtimes}^{\gamma_{1},\sigma_{1}}$
and the 1-D free limit theorems imply that 
\[
\rho_{1}\circ\pi_{1}^{-1}=\sigma_{1}=\rho_{1}^{\prime}\circ\pi_{1}^{-1}
\]
on $\mathbb{T}$. As will be seen below, these two observations will
lead to $\rho_{1}=\rho_{1}^{\prime}$.

First, for any closed subset $F\subset\mathbb{T}^{2}$, $1\notin\pi_{2}(F)$,
we have the distance 
\[
\min_{(s,t)\in F}|1-\pi_{2}(s,t)|>0,
\]
so that the function $(s,t)\mapsto1/(1-\Re t)$ is continuous on $F$.
Denoting 
\[
\rho=(1-\Re t)\, d\rho_{1}=(1-\Re t)\, d\rho_{1}^{\prime},
\]
we have 
\[
\rho_{1}(F)=\int_{F}\frac{1}{1-\Re t}\, d\rho=\rho_{1}^{\prime}(F)
\]
for all such $F$. Since finite Borel measures on $\mathbb{T}^{2}$
are Radon measures, we conclude that $\rho_{1}(E)=\rho_{1}^{\prime}(E)$
for any Borel subset $E\subset\mathbb{T}^{2}$ with $1\notin\pi_{2}(E)$
by approximation. 

Now, given a general Borel subset $E\subset\mathbb{T}^{2}$, we decompose
it into a disjoint union 
\[
E=[\underbrace{E\cap\pi_{2}^{-1}\{1\}}_{=E_{0}}]\cup[E\setminus E_{0}]
\]
of Borel measurable sets according to whether the second coordinate
is equal to one or not. Since $1\notin\pi_{2}(E\setminus E_{0})$,
the measures $\rho_{1}$ and $\rho_{1}^{\prime}$ agree on the sets
$E\setminus E_{0}$. If the set $E_{0}$ is empty then we have $\rho_{1}(E)=\rho_{1}^{\prime}(E)$.
If $E_{0}$ is not empty, we write the set $E_{0}$ as the product
$B\times\{1\}$ for some Borel measurable subset $B\subset\mathbb{T}$
and observe that
\begin{eqnarray*}
\rho_{1}(E_{0}) & = & \rho_{1}(B\times\mathbb{T})-\rho_{1}(B\times(\mathbb{T}\setminus\{1\}))\\
 & = & \sigma_{1}(B)-\rho_{1}^{\prime}(B\times(\mathbb{T}\setminus\{1\}))\\
 & = & \rho_{1}^{\prime}(B\times\mathbb{T})-\rho_{1}^{\prime}(B\times(\mathbb{T}\setminus\{1\}))=\rho_{1}^{\prime}(E_{0}),
\end{eqnarray*}
because $1\notin\pi_{2}(B\times(\mathbb{T}\setminus\{1\}))$. So we
still have $\rho_{1}(E)=\rho_{1}^{\prime}(E)$ in this case, proving
$\rho_{1}=\rho_{1}^{\prime}$. 

In the same way, we have 
\[
(1-\Re s)\, d\rho_{2}=(1-\Re s)\, d\rho_{2}^{\prime},\quad\rho_{2}\circ\pi_{2}^{-1}=\rho_{2}^{\prime}\circ\pi_{2}^{-1}=\sigma_{2}
\]
by analyzing the identity $\exp\left(f\cdot F_{\rho_{1},\rho_{2},a}\right)=\exp\left(f\cdot F_{\rho_{1}^{\prime},\rho_{2}^{\prime},a^{\prime}}\right)$
on the components $\mathbb{D}^{2}$ and $\Delta_{1}\times\mathbb{D}$.
It follows that $\rho_{2}=\rho_{2}^{\prime}$.

Finally, the real number $a$ is unique because 
\[
a=\rho(\mathbb{T}^{2})-\Re\left[f(0,0)F_{\rho_{1},\rho_{2},a}(0,0)\right]=\rho(\mathbb{T}^{2})-\Re\left[f(0,0)F_{\rho_{1}^{\prime},\rho_{2}^{\prime},a^{\prime}}(0,0)\right]=a^{\prime}.
\]

In conclusion, we have shown the convergence of the sequences $S$,
$T$, $C$, and therefore the system (\ref{eq:3.5}) and the exponential
integral representation $\Sigma_{\nu}=\exp\left(f\cdot F_{\rho_{1},\rho_{2},a}\right)$
are proved.

Conversely, assume (\ref{eq:3.5}) and note that it implies $\mu_{n}^{(j)}\Rightarrow\nu_{\boxtimes}^{\gamma_{j},\rho_{j}\circ\pi_{j}^{-1}}$
through the 1-D free limit theorems for $j=1,2$. In other words,
if we consider any weak probability limit $\nu$ for the bi-free convolutions
$\{\mu_{n}\}_{n=1}^{\infty}$, then its marginal law $\nu^{(j)}$
is uniquely determined by $\nu^{(j)}=\nu_{\boxtimes}^{\gamma_{j},\rho_{j}\circ\pi_{j}^{-1}}$
for $j=1,2$. From this, we also know that the limiting transform
$\Sigma_{\nu}$ is at least defined in a small bidisk centered at
the point $(0,0)$. 

Meanwhile, the proof of the ``only if'' part shows that near the
point $(0,0)$, the transform $\Sigma_{\nu}$ is uniquely determined
by the system (\ref{eq:3.5}) and $\Sigma_{\nu}=\exp\left(f\cdot F_{\rho_{1},\rho_{2},a}\right)$.
In particular, we have $\Sigma_{\nu}(0,0)=\exp\left(F_{\rho_{1},\rho_{2},a}(0,0)\right)\neq0$,
implying that the limit $\nu$ is in fact a member of $\mathscr{P}_{\mathbb{T}^{2}}^{\times}$
and $\Sigma_{\nu}$ is defined in a product set $\Omega_{r}$. Moreover,
the identity $\Sigma_{\nu}=\exp\left(f\cdot F_{\rho_{1},\rho_{2},a}\right)$
now holds everywhere in $\Omega_{r}$, meaning
that $\Sigma_{\nu}$ is globally determined by (\ref{eq:3.5}). By
Proposition 2.5, the limit point $\nu$ is unique and therefore $\mu_{n}\Rightarrow\nu$.
The proof of the theorem is now finished.
\end{proof}
The preceding proof shows that the parameters $\rho_{1}$, $\rho_{2}$,
$a$, $\gamma_{1}$, $\gamma_{2}$ in (\ref{eq:3.5}) are uniquely
associated with the limit law $\nu$, although $\nu$ may serve as
the weak limit for many infinitesimal arrays. 

We conclude this section with two examples. The first one is a bi-free
analogue of the wrapped Gaussian distribution in directional statistics. 
\begin{example}
Fix $r\in(0,\infty)$. For any $n>r$, let $\xi_{n}=\sqrt{1-r/n}+i\sqrt{r/n}$
and let $\mu_{n}$ be the law of a random vector $(X_{n},Y_{n})$,
where 
\[
(X_{n},Y_{n})=\begin{cases}
(\xi_{n},\xi_{n})\\
(\overline{\xi_{n}},\overline{\xi_{n}})
\end{cases}\quad\text{with equal probabilities.}
\]
We shall consider the infinitesimal array $\{\mu_{nk}\}_{n,k}$, where
$k_{n}=n$ and $\mu_{n1}=\mu_{n2}=\cdots=\mu_{nn}=\mu_{n}$. Choose
the centering parameter $\varepsilon=1$ so that $b_{nk}^{(1)}=1=b_{nk}^{(2)}$
and $\nu_{nk}=\mu_{n}$ for all $n\geq1$ and $1\leq k\leq n$. Then
for any $p,q\in\mathbb{Z}$, we have the convergence of the moment
\[
nE\left[X_{n}^{p}Y_{N}^{q}(1-\Re X_{n})\right]=nE\left[X_{n}^{p}Y_{N}^{q}(1-\Re Y_{n})\right]\rightarrow r/2=\int_{\mathbb{T}^{2}}s^{p}t^{q}\, d[(r/2)\delta_{(1,1)}](s,t)
\]
as $n\rightarrow\infty$, implying that $\rho_{1}=(r/2)\delta_{(1,1)}=\rho_{2}$.
Also, it is easy to see that 
\[
a=\lim_{n\rightarrow\infty}nE\left[\Im X_{n}\Im Y_{n}\right]=r
\]
and $\gamma_{j}=1$ for $j=1,2$. Thus, by Theorem 3.4, the limit
law $N(r)$ of the\emph{ $n$-fold bi-free convolution powers} $\mu_{n}^{\boxtimes\boxtimes n}=\mu_{n}\boxtimes\boxtimes\mu_{n}\boxtimes\boxtimes\cdots\boxtimes\boxtimes\mu_{n}$
exists, and it is determined by 
\[
\begin{cases}
\Sigma_{N(r)}=\exp\left(-r\cdot f\right) & \text{in }(\mathbb{C}\setminus\mathbb{T})^{2};\\
N(r)\circ\pi_{j}^{-1}=\nu_{\boxtimes}^{1,(r/2)\delta_{1}}, & j=1,2.
\end{cases}
\]
It is interesting to observe that the same kind of central limit process
with infinitesimal random vectors 
\[
(X_{n},Y_{n})=\begin{cases}
(\xi_{n},\overline{\xi_{n}})\\
(\overline{\xi_{n}},\xi_{n})
\end{cases}\quad\text{with equal probabilities}
\]
produces a limit distribution $N(-r)$ satisfying 
\[
\begin{cases}
\Sigma_{N(-r)}=\exp\left(r\cdot f\right) & \text{in }(\mathbb{C}\setminus\mathbb{T})^{2};\\
N(-r)\circ\pi_{j}^{-1}=\nu_{\boxtimes}^{1,(r/2)\delta_{1}}, & j=1,2.
\end{cases}
\]
In summary, the (non-degenerate) \emph{multiplicative bi-free normal
law} $N(a)$ for any $a\in\mathbb{R}\setminus\{0\}$ is a probability
measure in $\mathscr{P}_{\mathbb{T}^{2}}^{\times}$, whose $\Sigma$-transform
is simply $\exp(-a\cdot f)$ and its marginal laws are the usual free
unitary Brownian motion with L\'{e}vy parameters $\gamma_{1}=\gamma_{2}=1$
and $\sigma_{1}=\sigma_{2}=(|a|/2)\delta_{1}$. The degenerate case
$a=0$ corresponds to the point mass $\delta_{(1,1)}$.
\end{example}
The next example introduces an analogue of the Poisson law. 
\begin{example}
Given a probability measure $\mu\neq\delta_{(1,1)}$ on $\mathbb{T}^{2}$
and a parameter $r>0$, we set $k_{n}=n$ for $n>r$ and consider
the array 
\[
\mu_{n1}=\mu_{n2}=\cdots=\mu_{nn}=\mu_{n}=(1-r/n)\delta_{(1,1)}+(r/n)\mu.
\]
It is easy to see that $\rho_{1}=(1-\Re s)\, d[r\mu]$, $\rho_{2}=(1-\Re t)\, d[r\mu]$,
\[
a=\int_{\mathbb{T}^{2}}\Im s\Im t\, d[r\mu](s,t),\quad\text{and}\quad\gamma_{j}=\exp\left(-i\int_{\mathbb{T}}\Im x\, d[r\mu]^{(j)}(x)\right)\quad(j=1,2).
\]
Thus, the limit law $Poi(r,\mu)$ (called the \emph{multiplicative
bi-free compound Poisson law} of rate $r$ and jump distribution $\mu$)
has the marginal L\'{e}vy measure $\sigma_{j}=(1-\Re x)\, d[r\mu]^{(j)}$
for $j=1,2$, and the $\Sigma$-transform
\begin{eqnarray*}
\Sigma_{Poi(r,\mu)}(z,w) & = & \exp\left(f(z,w)\left\{ \int_{\mathbb{T}^{2}}\frac{1+zs}{1-zs}\frac{1+wt}{1-wt}(1-\Re t)(1-\Re s)d[r\mu](s,t)\right.\right.\\
 &  & -i\int_{\mathbb{T}^{2}}\frac{1+zs}{1-zs}(\Im t)(1-\Re s)\, d[r\mu](s,t)\\
 &  & \left.\left.-i\int_{\mathbb{T}^{2}}\frac{1+wt}{1-wt}(\Im s)(1-\Re t)\, d[r\mu](s,t)-\int_{\mathbb{T}^{2}}\Im s\Im t\, d[r\mu](s,t)\right\} \right)\\
 & = & \exp\left(\int_{\mathbb{T}^{2}}\frac{(1-zw)(1-s)(1-t)}{(1-zs)(1-wt)}\, d[r\mu](s,t)\right)
\end{eqnarray*}
in $(\mathbb{C}\setminus\mathbb{T})^{2}$. In view of the above construction,
we also have the degenerate cases
\[
Poi(r,\delta_{(1,1)})=\delta_{(1,1)}=Poi(0,\mu).
\]
 
\end{example}

\section{infinite divisibility}

We start with a standard definition. 
\begin{defn}
A probability measure $\nu$ on $\mathbb{T}^{2}$ is said to be \emph{$\boxtimes\boxtimes$-infinitely
divisible} if for each $n\geq2$, there exists a probability measure
$\nu_{n}$ on $\mathbb{T}^{2}$ such that $\nu=\nu_{n}^{\boxtimes\boxtimes n}$. 
\end{defn}
We aim to characterize the $\boxtimes\boxtimes$-infinitely divisible
measures in $\mathscr{P}_{\mathbb{T}^{2}}^{\times}$. 

Point masses are $\boxtimes\boxtimes$-infinitely divisible, because
$\delta_{\lambda}=\delta_{(\pi_{1}(\lambda)^{1/n},\,\pi_{2}(\lambda)^{1/n})}$.
More generally, since 
\[
\nu_{\boxtimes}^{\gamma_{1},\sigma_{1}}\otimes\nu_{\boxtimes}^{\gamma_{2},\sigma_{2}}=\left[\nu_{\boxtimes}^{\gamma_{1}^{1/n},\sigma_{1}/n}\otimes\nu_{\boxtimes}^{\gamma_{2}^{1/n},\sigma_{2}/n}\right]^{\boxtimes\boxtimes n},\qquad n\geq2,
\]
the product measure of two $\boxtimes$-infinitely divisible measures
from $\mathscr{P}_{\mathbb{T}}^{\times}$ is $\boxtimes\boxtimes$-infinitely
divisible. (There is no uniqueness for the convolution decomposition
in these examples, since any branch of the $n$-th root can be used
here as long as it makes sense.) Examples 3.5 and 3.6 show that 
\[
N(a)=N(a/n)^{\boxtimes\boxtimes n}\quad\text{and}\quad Poi(r,\mu)=Poi(r/n,\mu)^{\boxtimes\boxtimes n}.
\]
So, the multiplicative bi-free normal and compound Poisson laws are
also $\boxtimes\boxtimes$-infinitely divisible. Finally, any bi-free
convolution mixture of these measures remains $\boxtimes\boxtimes$-infinitely
divisible.

Weak limits of $\boxtimes\boxtimes$-infinitely divisible measures
are $\boxtimes\boxtimes$-infinitely divisible. Indeed, if $\nu_{n}\Rightarrow\nu$
and each $\nu_{n}$ is $\boxtimes\boxtimes$-infinitely divisible,
then to each fixed $m\geq2$ there is a sequence $\{\nu_{n}^{m}\}_{n=1}^{\infty}$
in $\mathscr{P}_{\mathbb{T}^{2}}$ such that $\nu_{n}=(\nu_{n}^{m})^{\boxtimes\boxtimes m}$.
The compactness of $\mathbb{T}^{2}$ implies that there exists a weak
limit point $\nu^{m}\in\mathscr{P}_{\mathbb{T}^{2}}$ for the sequence
$\{\nu_{n}^{m}\}_{n=1}^{\infty}$. We conclude that $\nu=(\nu^{m})^{\boxtimes\boxtimes m}$
after taking the limit, that is, the measure $\nu$ is $\boxtimes\boxtimes$-infinitely
divisible.

The next result establishes the universal role of $\boxtimes\boxtimes$-infinitely
divisible measures in limit theorems of infinitesimal arrays. Recall
that the notations $f$ and $F_{\rho_{1},\rho_{2},a}$ are the same
as in Theorem 3.4.
\begin{thm}
Given a measure $\nu\in\mathscr{P}_{\mathbb{T}^{2}}^{\times}$, the
following statements are equivalent:
\begin{enumerate}
\item The measure $\nu$ is $\boxtimes\boxtimes$-infinitely divisible.
\item There exist an infinitesimal array $\{\mu_{nk}\}_{n,k}\subset\mathscr{P}_{\mathbb{T}^{2}}^{\times}$
and a sequence $\lambda_{n}\in\mathbb{T}^{2}$ such that 
\[
\delta_{\lambda_{n}}\boxtimes\boxtimes\mu_{n1}\boxtimes\boxtimes\mu_{n2}\boxtimes\boxtimes\cdots\boxtimes\boxtimes\mu_{nk_{n}}\Rightarrow\nu.
\]

\end{enumerate}
\end{thm}
\begin{proof}
Assume that $\nu$ is $\boxtimes\boxtimes$-infinitely divisible,
and we write $\nu=\nu_{n}^{\boxtimes\boxtimes n}$ ($n\geq2$) so
that $\nu^{(j)}=[\nu_{n}^{(j)}]^{\boxtimes n}$ for $j=1,2$. Let
$\nu_{0}$ be a probability weak limit point of the first marginal
laws $\{\nu_{n}^{(1)}\}_{n=1}^{\infty}$ on $\mathbb{T}$. Passing
to a convergent subsequence if necessary, we have 
\[
\left|m(\nu_{0})\right|=\lim_{n\rightarrow\infty}\left|m(\nu_{n}^{(1)})\right|=\lim_{n\rightarrow\infty}\sqrt[n]{\left|m(\nu^{(1)})\right|}=1.
\]
The limit $\nu_{0}$ is therefore of the form $\delta_{c}$ for some
$c\in\mathbb{T}$. In other words, we can always find positive integers
$k_{n}\rightarrow\infty$ and a complex number $c_{1}\in\mathbb{T}$
such that $\delta_{c_{1}}\boxtimes\nu_{k_{n}}^{(1)}\Rightarrow\delta_{1}$
as $n\rightarrow\infty$. By a similar argument, we can further refine
the sequence $k_{n}$ so that the sequence $\delta_{c_{2}}\boxtimes\nu_{k_{n}}^{(2)}$
also tends weakly to $\delta_{1}$ for some $c_{2}\in\mathbb{T}$.
We then introduce a triangular array $\{\mu_{nk}\}_{n,k}$ and a sequence
$\{\lambda_{n}\}_{n=1}^{\infty}$ by 
\[
\mu_{n1}=\mu_{n2}=\cdots=\mu_{nk_{n}}=\delta_{(c_{1},c_{2})}\boxtimes\boxtimes\nu_{k_{n}},\quad\lambda_{n}=(c_{1}^{-k_{n}},c_{2}^{-k_{n}})\qquad n\geq1.
\]
The array $\{\mu_{nk}\}_{n,k}$ has identical rows, and it is infinitesimal
because 
\begin{multline*}
\mu_{nk}(\{(s,t)\in\mathbb{T}^{2}:|s-1|+|t-1|\geq\varepsilon\})\\
\leq\mu_{nk}^{(1)}(\{s\in\mathbb{T}:|s-1|\geq\varepsilon/2\})+\mu_{nk}^{(2)}(\{t\in\mathbb{T}:|t-1|\geq\varepsilon/2\}).
\end{multline*}
Moreover, we have 
\[
\nu=\delta_{\lambda_{n}}\boxtimes\boxtimes\mu_{n1}\boxtimes\boxtimes\mu_{n2}\boxtimes\boxtimes\cdots\boxtimes\boxtimes\mu_{nk_{n}},\qquad n\geq1,
\]
proving the statement (2).

For the converse, we assume (2) and write the marginal limit law $\nu^{(j)}=\nu_{\boxtimes}^{\gamma_{j},\sigma_{j}}$
for $j=1,2$. The system (\ref{eq:3.5}) shows that 
\begin{eqnarray}
\Sigma_{\nu}(z,w) & = & \exp\left(f(z,w)F_{\rho_{1},\rho_{2},a}(z,w)\right)\label{eq:4.1}\\
 & = & \lim_{n\rightarrow\infty}\exp\left(\int_{\mathbb{T}^{2}}\frac{(1-zw)(1-s)(1-t)}{(1-zs)(1-wt)}\, d\tau_{n}(s,t)\right)\nonumber 
\end{eqnarray}
for $(z,w)$ in $(\mathbb{C}\setminus\mathbb{T})^{2}$, where the
notation $\tau_{n}=\sum_{k=1}^{k_{n}}\nu_{nk}$.

If $\tau_{n}(\mathbb{T}^{2})=0$ for infinitely many $n$'s, then
the exponential integral form in (\ref{eq:4.1}) degenerates to one
infinitely often, which means that the map $\Sigma_{\nu}$ is constantly
one. In this case the limiting measure $\nu$ is of the form
\[
\nu=\nu_{\boxtimes}^{\gamma_{1},\sigma_{1}}\otimes\nu_{\boxtimes}^{\gamma_{2},\sigma_{2}},
\]
which is $\boxtimes\boxtimes$-infinitely divisible.

Alternatively, we assume that $r_{n}=\tau_{n}(\mathbb{T}^{2})>0$
for sufficiently large $n$. Introduce a new rotation sequence $\lambda_{n}^{\prime}\in\mathbb{T}^{2}$
by specifying its coordinates 
\[
\pi_{j}(\lambda_{n}^{\prime})=\pi_{j}(\lambda_{n})\exp\left(i\sum_{k=1}^{k_{n}}\arg b_{nk}^{(j)}\right),\qquad j=1,2,
\]
where the constants $b_{nk}^{(j)}$ refer back to (\ref{eq:3.1}).
Then the exponential integral form in (\ref{eq:4.1}) becomes the
$\Sigma$-tranform of the rotated compound Poisson laws $\delta_{\lambda_{n}^{\prime}}\boxtimes\boxtimes Poi(r_{n},\mu_{n})$,
where the jump distribution $\mu_{n}=\tau_{n}/r_{n}$. On the other
hand, by (\ref{eq:3.5}) again, we have the marginal weak convergence
\begin{eqnarray*}
[\delta_{\lambda_{n}^{\prime}}\boxtimes\boxtimes Poi(r_{n},\mu_{n})]\circ\pi_{j}^{-1} & = & \nu_{\boxtimes}^{\overline{\pi_{j}(\lambda_{n}^{\prime})}\,\exp\left(-i\int_{\mathbb{T}}\Im x\, d\tau_{n}^{(j)}(x)\right),\:(1-\Re x)\, d\tau_{n}^{(j)}}\\
 & \Rightarrow & \nu_{\boxtimes}^{\gamma_{j},\sigma_{j}}\qquad(j=1,2).
\end{eqnarray*}
So Theorem 2.10 implies $\delta_{\lambda_{n}^{\prime}}\boxtimes\boxtimes Poi(r_{n},\mu_{n})\Rightarrow\nu$,
meaning that $\nu$ is also bi-freely infinitely divisible in this
case.
\end{proof}
If $\nu\in\mathscr{P}_{\mathbb{T}^{2}}^{\times}$ is $\boxtimes\boxtimes$-infinitely
divisible, then the preceding result and Theorem 3.4 imply that $\Sigma_{\nu}=\exp\left(f\cdot F_{\rho_{1},\rho_{2},a}\right)$
where 
\[
(1-\Re t)\, d\rho_{1}=\text{w*-}\lim_{n\rightarrow\infty}(1-\Re s)(1-\Re t)\sum_{k=1}^{k_{n}}\nu_{nk}=(1-\Re s)\, d\rho_{2}.
\]
As we will see below, every exponential integral form like this corresponds
to a bi-freely infinitely divisible law. But of course there is no
uniqueness for this underlying infinitely divisible measure, due to
the very nature of the $\Sigma$-transform.
\begin{prop}
Let $\rho_{1}$ and $\rho_{2}$ be two measures in $\mathscr{M}_{\mathbb{T}^{2}}$
such that 
\[
(1-\Re t)\, d\rho_{1}=(1-\Re s)\, d\rho_{2},
\]
and let $a\in\mathbb{R}$. Then there exists a $\boxtimes\boxtimes$-infinitely
divisible measure $\nu\in\mathscr{P}_{\mathbb{T}^{2}}^{\times}$ such
that $\Sigma_{\nu}=\exp(f\cdot F_{\rho_{1},\rho_{2},a})$ in $(\mathbb{C}\setminus\mathbb{T})^{2}$. \end{prop}
\begin{proof}
Denote $\rho=(1-\Re t)\, d\rho_{1}=(1-\Re s)\, d\rho_{2}$ and decompose
the space $\mathbb{T}^{2}$ into the disjoint union $\mathbb{T}^{2}=(\mathbb{T}^{*}\times\mathbb{T}^{*})\cup(\mathbb{T}^{*}\times\{1\})\cup(\{1\}\times\mathbb{T}^{*})\cup\{(1,1)\}$
of Borel sets. Notice that 
\[
\rho((\{1\}\times\mathbb{T}^{*})\cup\{(1,1)\})=\int_{\{1\}\times\mathbb{T}}d\rho=\int_{\{1\}\times\mathbb{T}}(1-\Re s)\, d\rho_{2}(s,t)=0
\]
and 
\[
\rho(\mathbb{T}^{*}\times\{1\})=\int_{\mathbb{T}^{*}\times\{1\}}d\rho=\int_{\mathbb{T}^{*}\times\{1\}}(1-\Re t)\, d\rho_{1}(s,t)=0.
\]
So the measure $\rho$ only charges the set $\mathbb{T}^{*}\times\mathbb{T}^{*}$. 

Let $A_{n}=\{\exp(i\theta):1/n<|\theta|\leq\pi\}$ for $n\geq1$ and
observe that 
\[
\rho_{1}(\{1\}\times A_{n})=\int_{\{1\}\times A_{n}}d\rho_{1}=\int_{\{1\}\times A_{n}}\frac{1-\Re s}{1-\Re t}\, d\rho_{2}(s,t)=0.
\]
This shows that $\rho_{1}(\{1\}\times\mathbb{T}^{*})=\lim_{n\rightarrow\infty}\rho_{1}(\{1\}\times A_{n})=0$,
as well as $\rho_{2}(\mathbb{T}^{*}\times\{1\})=0$ by a similar argument.
Hence the measure $\rho_{1}$ places all its masses on the union $(\mathbb{T}^{*}\times\mathbb{T}^{*})\cup(\mathbb{T}\times\{1\})$,
while $\rho_{2}$ does the same on $(\mathbb{T}^{*}\times\mathbb{T}^{*})\cup(\{1\}\times\mathbb{T})$. 

If $\rho$ is the zero measure, we can further calculate 
\[
\rho_{1}(A_{n}\times A_{n})=\int_{A_{n}\times A_{n}}\frac{d\rho(s,t)}{1-\Re t}=0=\int_{A_{n}\times A_{n}}\frac{d\rho(s,t)}{1-\Re s}=\rho_{2}(A_{n}\times A_{n}),
\]
implying that $\rho_{1}(\mathbb{T}^{*}\times\mathbb{T}^{*})=0=\rho(\mathbb{T}^{*}\times\mathbb{T}^{*})$
after letting $n\rightarrow\infty$. It follows that both $\Im t\, d\rho_{1}$
and $\Im s\, d\rho_{2}$ are also the zero measure on $\mathbb{T}^{2}$,
and as a result, the exponential integral transform $\exp(f\cdot F_{\rho_{1},\rho_{2},a})$
here is simply $\exp(-a\cdot f)$, the $\Sigma$-transform of the
multiplicative bi-free normal law $N(a)$. We take $\nu=N(a)$ in
this case.

In the case of $\rho(\mathbb{T}^{2})=\rho(\mathbb{T}^{*}\times\mathbb{T}^{*})>0$,
the existence of $\nu$ will be shown by Poisson approximation as
in the proof of Theorem 4.2. Thus, we first re-write the given exponential
integral form $\Sigma=\exp(f\cdot F_{\rho_{1},\rho_{2},a})$ into
\begin{eqnarray*}
\Sigma(z,w) & = & \lim_{n\rightarrow\infty}\exp\left(f(z,w)\left\{ \int_{A_{n}\times A_{n}}\frac{1+zs}{1-zs}\frac{1+wt}{1-wt}\, d\rho(s,t)\right.\right.\\
 &  & \left.\left.-i\int_{A_{n}\times A_{n}}\frac{(1+zs)\Im t}{1-zs}\, d\rho_{1}(s,t)-i\int_{A_{n}\times A_{n}}\frac{(1+wt)\Im s}{1-wt}\, d\rho_{2}(s,t)-a\right\} \right)\\
 & = & \lim_{n\rightarrow\infty}\Sigma_{Poi(r_{n},\mu_{n})\boxtimes\boxtimes N(a-a_{n})}(z,w),
\end{eqnarray*}
where the Poisson parameters 
\[
r_{n}=\int_{A_{n}\times A_{n}}\frac{d\rho(s,t)}{(1-\Re s)(1-\Re t)},\quad\mu_{n}=\frac{1}{r_{n}(1-\Re s)(1-\Re t)}\, I_{A_{n}\times A_{n}}(s,t)\, d\rho,
\]
$I_{A_{n}\times A_{n}}$ is the indicator function of the set $A_{n}\times A_{n}$,
and the auxiliary integral 
\[
a_{n}=\int_{\mathbb{T}^{2}}\Im s\Im t\, d[r_{n}\mu_{n}](s,t).
\]
Note that the total mass $r_{n}$ is strictly positive for sufficiently
large $n$. As always, the compactness of $\mathbb{T}^{2}$ yields
a weak probability limit point $\nu$ for the sequence $\nu_{n}=Poi(r_{n},\mu_{n})\boxtimes\boxtimes N(a-a_{n})$.
This limit point $\nu$ is $\boxtimes\boxtimes$-infinitely divisible
and satisfies the desired relation $\Sigma_{\nu}=\exp(f\cdot F_{\rho_{1},\rho_{2},a})$.
Moreover, if $\nu_{n}\Rightarrow\nu$ along a subsequence $A\subset\mathbb{N}$,
we have 
\begin{eqnarray*}
m_{1,1}(\nu)/[m(\nu^{(1)})m(\nu^{(2)})] & = & \lim_{n\rightarrow\infty,\, n\in A}\, m_{1,1}(\nu_{n})/[m(\nu_{n}^{(1)})m(\nu_{n}^{(2)})]\\
 & = & \lim_{n\rightarrow\infty,\, n\in A}\,\Sigma_{\nu_{n}}(0,0)\\
 & = & \Sigma(0,0)=\exp(F_{\rho_{1},\rho_{2},a}(0,0))\neq0,
\end{eqnarray*}
showing that $\nu\in\mathscr{P}_{\mathbb{T}^{2}}^{\times}$.
\end{proof}
Finally, Proposition 2.3 shows that $\text{m}\boxtimes\boxtimes\text{m}=\text{m}$,
and hence the uniform distribution $\text{m}$ is also $\boxtimes\boxtimes$-infinitely
divisible. In order to see that $\text{m}$ can serve as the weak limit
of an infinitesimal array, we end this paper with the following convergence
criteria.
\begin{prop}
Let $\{\mu_{n}\}_{n=1}^{\infty}$ be a sequence in $\mathscr{P}_{\mathbb{T}^{2}}^{\times}$,
and let $\{k_{n}\}_{n=1}^{\infty}$ be an unbounded sequence in $\mathbb{N}$.
The measures $\delta_{\lambda_{n}}\boxtimes\boxtimes\mu_{n}^{\boxtimes\boxtimes k_{n}}$
converge weakly to $\text{\emph{m}}$ for some (and hence for any)
$\lambda_{n}\in\mathbb{T}^{2}$ if and only if 
\[
\begin{cases}
[m_{1,1}(\mu_{n})]^{k_{n}}\rightarrow0;\\
{}[m(\mu_{n}^{(j)})]^{k_{n}}\rightarrow0, & j=1,2.
\end{cases}
\]
\end{prop}
\begin{proof}
Denote $\nu_{n}=\delta_{\lambda_{n}}\boxtimes\boxtimes\mu_{n}^{\boxtimes\boxtimes k_{n}}$,
we first show the necessity of the moment conditions. Indeed, we have
\begin{eqnarray*}
\left|m_{1,1}(\mu_{n})^{k_{n}}\right| & = & \left|\Sigma_{\mu_{n}}(0,0)^{k_{n}}m(\mu_{n}^{(1)})^{k_{n}}m(\mu_{n}^{(2)})^{k_{n}}\right|\\
 & = & \left|\Sigma_{\nu_{n}}(0,0)m(\nu_{n}^{(1)})m(\nu_{n}^{(2)})\right|=\left|m_{1,1}(\nu_{n})\right|\rightarrow\left|m_{1,1}(\text{m})\right|=0
\end{eqnarray*}
as $n\rightarrow\infty$. The second moment condition follows from
the identity $|m(\mu_{n}^{(j)})^{k_{n}}|=|m(\nu_{n}^{(j)})|$.

Conversely, assume $[m_{1,1}(\mu_{n})]^{k_{n}},\,[m(\mu_{n}^{(j)})]^{k_{n}}\rightarrow0$
for $j=1,2$. Let $\nu$ be any weak limit point of $\{\nu_{n}\}_{n=1}^{\infty}$
in $\mathscr{P}_{\mathbb{T}^{2}}$, and let $\ell_{n}$ be the largest
positive integer that is less than or equal to $k_{n}/2$. Since $\ell_{n}\approx k_{n}/2$
and $k_{n}-\ell_{n}\approx k_{n}/2$ as $n\rightarrow\infty$, we
have 
\begin{equation}
\begin{cases}
[m_{1,1}(\mu_{n})]^{\ell_{n}},\,[m_{1,1}(\mu_{n})]^{k_{n}-\ell_{n}}\rightarrow0;\\
{}[m(\mu_{n}^{(j)})]^{\ell_{n}},\,\,[m(\mu_{n}^{(j)})]^{k_{n}-\ell_{n}}\rightarrow0, & j=1,2.
\end{cases}\label{eq:4.2}
\end{equation}
Now consider the sequences $\delta_{\lambda_{n}}\boxtimes\boxtimes\mu_{n}^{\boxtimes\boxtimes\ell_{n}}$
and $\mu_{n}^{\boxtimes\boxtimes(k_{n}-\ell_{n})}$ and assume, after
passing to convergent subsequences, that they converge weakly to probability
measures $\rho_{1}$ and $\rho_{2}$, respectively. So, we have $\nu=\rho_{1}\boxtimes\boxtimes\rho_{2}$,
where the condition (\ref{eq:4.2}) implies that $m(\rho_{1}^{(j)})=0=m(\rho_{2}^{(j)})$
for $j=1,2$ and that $m_{1,1}(\rho_{1})=0=m_{1,1}(\rho_{2})$. This
means $\nu=\text{m}$ by Proposition 2.3. Since $\{\nu_{n}\}_{n=1}^{\infty}$
has only one weak limit point $\text{m}$, it follows that $\nu_{n}\Rightarrow\text{m}$.\end{proof}

\subsection*{Acknowledgement}The first author was supported through a grant from the Ministry of Science and Technology in Taiwan and a faculty startup grant from the National Sun Yat-sen University. The second author was supported by the NSERC Canada Discovery Grant RGPIN-2016-03796.

\end{document}